\documentclass[11pt]{amsart}
\usepackage{amssymb,latexsym}

\def\be1{{\begin{equation}}}
\def\ee1{{\end{equation}}}

\def\part{\partial}

\newtheorem{them}{Theorem}[section]

\newtheorem{corollary}{Corollary}[section]

\newtheorem{definition}[them]{Definition}

\numberwithin{equation}{section}

\newtheorem{lemma}{Lemma}[section]
\newtheorem{proposition}[lemma]{Proposition}
\newtheorem{theorem}[lemma]{Theorem}

\title{Stabilities of homothetically shrinking Yang-Mills solitons}

\author{Zhengxiang Chen}
\address{ Albert-Ludwigs-Universit\"at Freiburg, Mathematisches Institut, Eckerstr. 1, 79104 Freiburg, Germany;
New Address: Institute of Mathematics, Academy of Mathematics and Systems Science,
Chinese Academy of Sciences, Beijing 100190, China}
\email{zx.chen@amss.ac.cn}

\author{Yongbing Zhang}
\address{School of Mathematical Sciences and Wu Wen-Tsun Key Laboratory of Mathematics,
University of Science and Technology of China, Hefei 230026, Anhui Province, China}
\email{ybzhang@amss.ac.cn}
\thanks{The project is supported by NSFC No. 11201448}


\begin{document}
\maketitle
\begin{abstract}
In this paper we introduce  entropy-stability and F-stability for homothetically shrinking Yang-Mills solitons,
employing entropy and second variation of $\mathcal{F}$-functional respectively.
For a homothetically shrinking soliton which does not descend,
we prove that entropy-stability implies F-stability.
These stabilities have connections with the study of Type-I singularities of the Yang-Mills flow.
Two byproducts are also included:
We show that the Yang-Mills flow in dimension four cannot develop a Type-I singularity;
and we obtain a gap theorem for homothetically shrinking solitons.
\end{abstract}

\section{Introduction}

In this paper we introduce entropy-stability and F-stability for homothetically shrinking (Yang-Mills) solitons.
Let $E$ be a trivial $G$-vector bundle over $\mathbb{R}^n$ and of rank $r$.
Here the gauge group $G$ is a Lie subgroup of $SO(r)$.
A homothetically shrinking soliton, centered at the space-time point $(x_0=0, t_0=1)$, is a connection $A(x)$ on $E$ such that
\begin{equation}\nonumber
(d^\nabla)^*F+\frac{1}{2}i_{x}F=0,
\end{equation}
where $F$ is the curvature of $A(x)$, $(d^\nabla)^*$ denotes the formal adjoint of the covariant exterior differentiation $d^\nabla$,
and $i_x$ stands for the interior product by the position vector $x$.

A homothetically shrinking soliton $A(x)$ gives rise to a special solution of the Yang-Mills flow.
In fact in the exponential gauge of $A(x)$, the following
$$A(x,t)=A_j(x,t)dx^j:=(1-t)^{-\frac{1}{2}}A_j((1-t)^{-\frac{1}{2}}x)dx^j$$
is a solution to the Yang-Mills flow. On the other hand, homothetically shrinking solitons are closely related to Type-I singularities
of the Yang-Mills flow. Weinkove \cite{Weinkove} proved that Type-I singularities of the Yang-Mills flow are modelled by
homothetically shrinking solitons whose curvatures do not vanish identically.
Examples of homothetically shrinking solitons have been found in \cite{Gastel,Weinkove}.
In this paper, we restrict ourselves to homothetically shrinking solitons which have uniform bounds on $|\nabla^kA(x)|$ for each $k\geq 1$.
In fact, Weinkove showed in \cite{Weinkove} that Type-I singularities of the Yang-Mills flow can be modelled by such solitons.

Recently, Colding and Minicozzi \cite{ColdingMinicozzi} discovered two functionals for immersed surfaces in Euclidean space,
i.e. the $\mathcal{F}$-functional and the entropy. Critical points of both functionals are self-shrinkers of the mean curvature flow.
Colding and Minicozzi introduced entropy-stability and F-stability for self-shrinkers.
Inspired by their work, in this paper we aim to introduce corresponding stabilities for homothetically shrinking Yang-Mills solitons.
In fact there are many aspects in common concerning the entropy-stability and F-stability for self-similar solutions to various geometric flows,
which includes mean curvature flow, Ricci flow, harmonic map heat flow, and Yang-Mills flow.
For the entropy-stability and linearly stability of Ricci solitons, see for instance \cite{CaoHamiltonIlmanen, CaoZhu};
for the entropy-stability and F-stability of self-similar solutions to the harmonic map heat flow see \cite{Zhang}.

We begin with the definition of $\mathcal{F}$-functional.
Let $x_0$ be a point in $\mathbb{R}^n$ and $t_0$ a positive number.
The $\mathcal{F}$-functional with respect to $(x_0,t_0)$, defined on the space of connections on $E$, is given by
\begin{equation}\label{Ffunctional0}
\mathcal{F}_{x_0,t_0}(A)=t_0^2\int_{\mathbb{R}^n}|F|^2(4\pi t_0)^{-\frac{n}{2}}e^{-\frac{|x-x_0|^2}{4t_0}}dx.
\end{equation}

The functional $\mathcal{F}_{x_0,t_0}$ can trace back to the monotonicity formula of the Yang-Mills flow.
For the monotonicity formula see \cite{ChenShen, Hamilton, Naito}.
Let $A(x,t)$ be a solution to the Yang-Mills flow on $E$ and
\begin{equation}\nonumber
\Phi_{x_0,t_0}(A(x,t))=(t_0-t)^2\int_{\mathbb{R}^n}|F|^2[4\pi (t_0-t)]^{-\frac{n}{2}}e^{-\frac{|x-x_0|^2}{4(t_0-t)}}dx.
\end{equation}
Along the Yang-Mills flow, $\Phi_{x_0,t_0}$ is non-increasing in $t$.
Moreover $\Phi_{x_0,t_0}$ is preserved if and only if $A(x,0)$ is a homothetically shrinking soliton centered at $(x_0,t_0)$.
Here a homothetically shrinking soliton centered at $(x_0,t_0)$ is a connection on $E$ satisfying the equation
\begin{equation}\label{homotheticallyshrinker0}
(d^\nabla)^*F+\frac{1}{2t_0}i_{x-x_0}F=0.
\end{equation}

The $\mathcal{F}$-functional leads to another characterization of homothetically shrinking solitons:
Critical points of $\mathcal{F}_{x_0,t_0}$ are exactly homothetically shrinking solitons centered at $(x_0,t_0)$;
moreover, $(x_0,t_0, A_0)$ is a critical point of the function $(x,t,A)\mapsto \mathcal{F}_{x,t}(A)$
if and only if $A_0$ is a homothetically shrinking soliton centered at $(x_0,t_0)$.

The $\lambda$-entropy of a connection $A(x)$ on the bundle $E$ is defined by
\begin{equation}\label{entropy0}
\lambda(A)=\sup_{x_0\in \mathbb{R}^n,t_0>0}\mathcal{F}_{x_0,t_0}(A).
\end{equation}
A crucial fact is the following
\begin{proposition}
Let $A(x,t)$ be a solution to the Yang-Mills flow on the bundle $E$.
Then the entropy $\lambda(A(x,t))$ is non-increasing in $t$.
\end{proposition}
The entropy is a rescaling invariant. More precisely, let $A(x)$ be a connection on $E$ and $A^c$ a rescaling of $A(x)$ given by
$A_i^c(x)=c^{-1}A_i(c^{-1}x)$, $c>0$.
Then
$\mathcal{F}_{cx_0,c^2t_0}(A^c)=\mathcal{F}_{x_0,t_0}(A)$
and hence $\lambda(A^c)=\lambda(A)$.
In particular the entropy of each time-slice of the homothetically shrinking Yang-Mills flow, induced from a homothetically shrinking soliton,
is preserved. The entropy is also invariant under translations of a connection. Let $A(x)$ be a connection on $E$,
$x_1\in \mathbb{R}^n$ a given point, and $\widetilde{A}_i(x)=A_i(x+x_1)$.
Then we have $\mathcal{F}_{x_0-x_1,t_0}(\widetilde{A})=\mathcal{F}_{x_0,t_0}(A)$ and hence $\lambda(\widetilde{A})=\lambda(A)$.

In general the entropy $\lambda(A)$ of a connection $A(x)$ is not attained by any $\mathcal{F}_{x_0,t_0}(A)$.
However if $A(x)$ is a homothetically shrinking soliton centered at $(x_0,t_0)$,
then $\lambda(A)=\mathcal{F}_{x_0,t_0}(A)$. In fact we prove the following

\begin{proposition}\label{entropyobtained0}
Let $A(x)$ be a homothetically shrinking soliton centered at $(0,1)$ such that $i_VF\neq 0$ for any non-zero $V\in \mathbb{R}^n$.
Then the function $(x_0,t_0)\mapsto F_{x_0,t_0}(A)$ attains its strict maximum at $(0,1)$.
\end{proposition}
Note that if $i_VF=0$ for some non-zero vector $V$, then $A(x)$ can be viewed as a connection on a $G$-vector bundle over any hyperplane perpendicular
to $V$ and we say $A(x)$ descends (to $V^\perp$).

Entropy-stability and F-stability are defined for homothetically shrinking solitons.
\begin{definition}
A homothetically shrinking soliton $A(x)$ is called entropy-stable if it is a local minimum of the entropy, among
all perturbations $\widetilde{A}(x)$ such that $||\widetilde{A}-A||_{C^1}$ is sufficiently small.
\end{definition}
Entropy-stability of homothetically shrinking solitons has direct connections with Type-I singularities of the Yang-Mills flow.
For example, given an entropy-unstable homothetically shrinking soliton $A(x)$, by definition
we can find a perturbation $\widetilde{A}(x)$ of $A(x)$ such that  $||\widetilde{A}-A||_{C^1}$ is arbitrarily small and has less entropy.
Then by comparing the entropy, the Yang-Mills flow starting from $\widetilde{A}$ cannot converge back to a rescaling of $A(x)$.
Moreover, the Yang-Mills flow cannot develop a Type-I singularity modelled by $A(x)$, due to the fact that the entropy is a rescaling invariant.

Let $A_0(x)$ be a homothetically shrinking soliton centered at $(x_0,t_0)$.
For a $1$-parameter family of deformations $(x_s, t_s, A_s)$ of $(x_0,t_0,A_0)$, let
$V=\frac{dx_s}{ds}|_{s=0}, q=\frac{dt_s}{ds}|_{s=0}, \theta=\frac{dA_s}{ds}|_{s=0}$.
\begin{definition} $A_0(x)$ is called F-stable
if for any compactly supported $\theta$,
there exist a real number $q$ and a vector $V$ such that
$$\mathcal{F}_{x_0,t_0}''(q,V,\theta):=\frac{d^2}{ds^2}|_{s=0}\mathcal{F}_{x_s,t_s}(A_s)\geq 0.$$
\end{definition}

Entropy-stability has an apparent connection with the singular behavior of the Yang-Mills flow;
however the F-stability is more practical when we are trying to do classification.
The classification of entropy-stable homothetically shrinking solitons can be relied
on the classification of F-stable ones. In fact we have the following relation for entropy-stability and
F-stability.

\begin{theorem}\label{entropyimplyF}
Let $A(x)$ be a homothetically shrinking soliton such that $i_VF\neq 0$ for any non-zero $V\in \mathbb{R}^n$.
If $A(x)$ is entropy-stable, then it is F-stable.
\end{theorem}

Let $A_0(x)$ be a homothetically shrinking soliton centered at $(0,1)$. Denote
\begin{equation}\label{Jacobioperator0}
L\theta=-[(d^\nabla)^*d^\nabla\theta+\mathcal{R}(\theta)+i_{\frac{x}{2}}d^\nabla\theta],
\end{equation}
where $\mathcal{R}(\theta)(\partial_j):=[F_{ij},\theta_i]$. For the homothetically shrinking soliton $A_0(x)$, we have
\begin{equation} \label{J0}
L (d^\nabla)^*F=(d^\nabla)^*F
\end{equation}
and
$$Li_VF=\frac{1}{2}i_VF, \quad \forall V\in \mathbb{R}^n.$$
The second variation of the $\mathcal{F}$-functional and at $A_0$ is given by
\begin{equation}\label{secondvariationformula0}
\frac{1}{2}\mathcal{F}_{0,1}''(q,V,\theta)=\int_{\mathbb{R}^n}<-L\theta+2q(d^\nabla)^*F-i_VF, \theta>G dx
-\int_{\mathbb{R}^n}(q^2|(d^\nabla)^*F|^2+\frac{1}{2}|i_V F|^2)G dx,
\end{equation}
where $G(x)=(4\pi)^{-\frac{n}{2}}\exp(-\frac{|x|^2}{4})$.
Denote the space of $\theta$ satisfying $L\theta=-\lambda\theta$ by $E_\lambda$.
We have the following characterization for F-stability.
\begin{theorem}\label{FstabilityandL}
$A_0(x)$ is F-stable if and only if the following properties are satisfied
\begin{itemize}
\item $E_{-1} = \{ c(d^\nabla)^*F, \quad c\in \mathbb{R}\}  $;
\item $E_{- \frac{1}{2}} = \{ i_V F, \quad V \in \mathbb{R}^n \} $;
\item $E_{\lambda} = \{0\},  \mbox{ for any }  \lambda <0   \mbox{ and }  \lambda \neq -1,\, - \frac{1}{2} .$
\end{itemize}
\end{theorem}

Theorem \ref{FstabilityandL} amounts to say that $A_0(x)$ is F-stable if and only if $L$ is non-negative definite modulo
the vector space spanned by $(d^\nabla)^*F$ and $i_VF$.  This is actually the reflection of the invariance property
of the $\mathcal{F}$-functional and the entropy under rescalings and translations.
Since Colding-Minicozzi's work \cite{ColdingMinicozzi},
classification problem of F-stable self-shrinkers of the mean curvature flow has drawn much attention, see for instance
\cite{AndrewsLiWei, ArezzoSun, LeeLue, LiZhang}.

We have two simple byproducts regarding homothetically shrinking solitons.
We show the non-existence of homothetically shrinking solitons in dimensions four and lower, and a gap theorem.
Let $A(x)$ be a homothetically shrinking soliton centered at $(0,1)$. Then we have the identity
$$\int_{\mathbb{R}^n}|x|^2|F|^2G(x)dx=2(n-4)\int_{\mathbb{R}^n}|F|^2G(x) dx.$$
It immediately implies the following
\begin{proposition}\label{notypeI0}
When $n=2, 3$, or $4$, there exists no homothetically shrinking soliton
such that $|F|$ is uniformly bounded and not identically zero.
\end{proposition}

R{\aa}de \cite{Rade} proved that the Yang-Mills flow, over a compact Riemannian manifold of dimension $n=2$ or $3$,
exists for all time and converges to a Yang-Mills connection.
However if the base manifold has dimension five or above, Naito \cite{Naito} showed that
the Yang-Mills flow can develop a singularity in finite time, see also \cite{Grotowski}.
It is unclear yet whether the Yang-Mills flow over a four-dimensional manifold develops a singularity in finite time.
For partial results in this dimension, see for instance \cite{Hong,Schlatter,Struwe}
{\sl , the remarkable monographs \cite{Feehan} and the references therein}.
Together with Weinkove's blowup analysis for Type-I singularities of the Yang-Mills flow,
Proposition \ref{notypeI0} shows that the Yang-Mills flow cannot develop a singularity of Type-I.
This was actually a known fact, see for instance \cite{Hamilton}.

Gap theorems for Yang-Mills connections over spheres was considered in \cite{BourguignonLawson}.
Gap theorems for various kinds of self-similar solutions have also been obtained, see for instance \cite{CaoLi, LeSesum, Zhang}.
By (\ref{J0}), we have the following gap result for homothetically shrinking solitons.

\begin{theorem}\label{gap0}
Let $A(x)$ be a homothetically shrinking soliton centered at $(0,1)$.
If $ |F|^2 < \frac{n}{2(n-1)} ,$
then $(E, A)$ is flat.
\end{theorem}

The paper is organized as follows: in the next section we review some background,
with emphasis on homothetically shrinking solitons and
Weinkove's blowup analysis for Type-I singularities of the Yang-Mills flow.
In Section 3, we consider the $\mathcal{F}$-functional and its first variation.
Section 4 is devoted to the calculation of the second variation of the $\mathcal{F}$-functional,
i.e. (\ref{secondvariationformula0}).
In Section 5, we study the F-stability of homothetically shrinking solitons and prove
Theorem \ref{FstabilityandL} and Theorem \ref{gap0}.
In Section 6, we introduce the $\lambda$-entropy and prove Proposition \ref{entropyobtained0}.
In the last section, we prove that entropy-stability implies F-stability, i.e. Theorem \ref{entropyimplyF}.

We would like to point out that although we assume, for simplicity, that the homothetically shrinking solitons
have uniform bounds on $|\nabla^kA|$, our statements except Theorem \ref{entropyimplyF} are still straightforwardly valid if
$|\nabla^kA|$ has polynomial growth. 
{\sl Many results in this paper have also been obtained by Kelleher and Streets \cite{KelleherStreets} independently.}

\section{Preliminaries}

In this section we briefly introduce the Yang-Mills flow and its singularity.
We shall introduce the blowup analysis for Type-I singularities, which was carried out by Weinkove \cite{Weinkove}.
It leads to the main object in this paper, i.e. homothetically shrinking soliton.

Let $(M,g)$ be a closed $n$-dimensional Riemannian manifold.
Let $G$ be a compact Lie group and $P(M,G)$ a principle bundle over $M$ with the structure group $G$.
We fix a $G$-vector bundle $E_M=P(M,G)\times_\rho \mathbb{R}^r$,
associated to $P(M,G)$ via a faithful representation $\rho:G\rightarrow SO(r)$.
Let $\mathfrak{g}$ denote the Lie algebra of $G$.
A connection on $E_M$ is locally a $\mathfrak{g}$-valued $1$-form.
Using Latin letters for the manifold indices,
one may write a connection $A$ in the form of $A=A_idx^i$,
where $A_i\in so(r)$. Using Greek letters for the bundle indices, one may also write
$A=A_{i\beta}^\alpha dx^i.$
The curvature of the connection $A$ is locally a $\mathfrak{g}$-valued $2$-form
$$F=\frac{1}{2}F_{ij}dx^i\wedge dx^j=\frac{1}{2}F_{ij\beta}^\alpha dx^i\wedge dx^j,$$
and
$$F_{ij}=\partial_iA_j-\partial_jA_i+[A_i,A_j].$$

The Yang-Mills functional, defined on the space of connections, is given by
\begin{equation}\label{YMfunctional}
YM(A)=\frac{1}{2}\int_M|F|^2d\mu_g,
\end{equation}
where
$$|F|^2=\frac{1}{2}g^{ik}g^{jl}<F_{ij},F_{kl}>=\frac{1}{2}g^{ik}g^{jl}F_{ij\beta}^\alpha F_{kl\beta}^\alpha,$$
and
$$F_{ij\beta}^\alpha=\partial_iA_{j\beta}^\alpha-\partial_jA_{i\beta}^\alpha
+A_{i\gamma}^\alpha A_{j\beta}^\gamma-A_{j\gamma}^\alpha A_{i\beta}^\gamma.$$
Let $\nabla$ denote the covariant differentiation on $\Gamma(E_M)$ associated to the connection $A$,
and also the covariant differentiation on $\mathfrak{g}$-valued $p$-forms induced by $A$ and the Levi-Civita connection of $(M,g)$.
Curvature $F$ satisfies the Bianchi identity
$d^\nabla F=0$, where $d^\nabla$ denotes the covariant exterior differentiation.
Let $(d^\nabla)^*$ denote the formal adjoint of $d^\nabla$.
A connection $A$ is a critical point of the Yang-Mills functional, called a Yang-Mills connection, if and only if  it is a solution of the
Yang-Mills equation $(d^\nabla)^*F=0$.
The Yang-Mills equation can also be written as
$$\nabla^pF_{pj\beta}^\alpha=0.$$
In normal coordinates of $(M,g)$, we have $\nabla^pF_{pj\beta}^\alpha=\partial_pF_{pj\beta}^\alpha+A_{p\gamma}^\alpha F_{pj\beta}^\gamma-F_{pj\gamma}^\alpha A_{p\beta}^\gamma.$

As the $L^2$-gradient flow of the Yang-Mills functional, the Yang-Mills flow is defined by
\begin{equation}\label{YMflow}
\frac{dA}{dt}=-(d^\nabla)^*F.
\end{equation}
Assume $A(x,t)$ is a smooth solution to the Yang-Mills flow for $0\leq t<T$ and as $t\rightarrow T$ the curvature blows up, i.e.
$\limsup_{t\rightarrow T}\max_{x\in M}|F(x,t)|=\infty$.
If there exists a positive constant $C$ such that
\begin{equation}\label{typeI}
|F(x,t)|\leq \frac{C}{T-t},
\end{equation}
one says that the Yang-Mills flow develops a Type-I singularity, or a rapidly forming singularity.
Otherwise one says that the Yang-Mills flow develops a Type-II singularity.
If (\ref{typeI}) is satisfied and $x_0$ is a point such that $\limsup_{t\rightarrow T}|F(x_0,t)|=\infty$,
we call $(x_0,T)$ a Type-I singularity.

Let $A(x,t)$ be a smooth solution to the Yang-Mills flow and $(x_0,T)$ a Type-I singularity.
We now follow \cite{Weinkove} introducing the blowup procedure around $(x_0,T)$.
Let $B_r(x_0)$ be a small geodesic ball centered at $x_0$ and of radius $r$ over which $E_M$ is trivial.
For simplicity we identify $B_r(x_0)$ with the ball $B_r(0)$ in $\mathbb{R}^n$.
Let $\lambda_i$ be a sequence of positive numbers tending to zero.
For each $i$, one gets a Yang-Mills flow $A^{\lambda_i}(y,s)$ by setting
\begin{equation}\label{rescaling}
A^{\lambda_i}(y,s)=\lambda_iA_p(\lambda_iy,T+\lambda_i^2s)dy^p,
\quad y\in B_{r/\lambda_i}(0), s\in [-\lambda_i^{-2}T,0).
\end{equation}
(An alternative way of obtaining a sequence of blowups of $A(x,t)$ is to rescale the metric around the singular point $x_0$.)
Let $x=\lambda_iy$ and $t=T+\lambda_i^2s$.
By the assumption (\ref{typeI}), the curvature of $A^{\lambda_i}$ satisfies
$$|F^{\lambda_i}(y,s)|=\lambda_i^2|F(x,t)|=|s|^{-1}(T-t)|F(x,t)|\leq C|s|^{-1}.$$

Let $h=h^\alpha_\beta$ be a gauge transformation which acts on connections by
$$h^*\nabla=h^{-1}\circ \nabla\circ h,$$
or equivalently,
$$h^*A=h^{-1}dh+h^{-1}Ah.$$
Note that gauge transformations preserve Yang-Mills flows.
Hence $h^*A^{\lambda_i}(y,s)$ defines a solution to the Yang-Mills flow.
Weinkove \cite{Weinkove} proved the following

\begin{theorem}\label{Weinkove}
Let $(x_0,T)$ be a Type-I singularity of the Yang-Mills flow $A(x,t)$ over $M$.
Then there exists a sequence of blowups $A^{\lambda_i}(y,s)$ defined by (\ref{rescaling}) and a sequence of gauge transformations $h_i$ such that
$h_i^*A^{\lambda_i}(y,s)$ converges smoothly on any compact set to a flow $\widetilde{A}(y,s)$.
Here $\widetilde{A}(y,s)$, defined on a trivial $G$-vector bundle over $\mathbb{R}^n\times (-\infty,0)$, is a solution to the Yang-Mills flow,
which has non-zero curvature and satisfies
\begin{equation}\label{limiting}
\widetilde{\nabla}^p\widetilde{F}_{pj}-\frac{1}{2|s|}y^p\widetilde{F}_{pj}=0.
\end{equation}
\end{theorem}

In Theorem \ref{Weinkove}, $h_i$ are chosen as suitable Coulomb gauge transformations so that for any $s<0$ and $k\geq 1$,
$|\nabla^kh_i^*A^{\lambda_i}|$ is uniformly bounded. The bounds do not depend on $i$.
Hence for any $s<0$ and $k\geq 1$, $|\nabla^k\widetilde{A}|$ is uniformly bounded.

A solution $A(y,s)$ to the Yang-Mills flow, defined on a trivial bundle over $\mathbb{R}^n\times (-\infty,0)$, is called a homothetically
shrinking soliton if it satisfies
\begin{equation}\label{homotheticallyshrinker}
A_{i\beta}^\alpha(y,s)=\frac{1}{\sqrt{|s|}} A_{i\beta}^\alpha(\frac{y}{\sqrt{|s|}}, -1)
\end{equation}
for any $y\in \mathbb{R}^n$ and $s<0$, for more details see \cite{Weinkove}.
The limiting Yang-Mills flow $\widetilde{A}(y,s)$ is actually a homothetically shrinking soliton.
In fact via an exponential gauge for $\widetilde{A}(y,s)$, in which $y^p\widetilde{A}_{p\beta}^\alpha=0$,
(\ref{limiting}) and (\ref{homotheticallyshrinker}) are equivalent for the Yang-Mills flow $\widetilde{A}(y,s)$.

One of the main ingredients of Theorem \ref{Weinkove} is the monotonicity formula for the Yang-Mills flow,
see \cite{ChenShen,Hamilton,Naito}.
In the simplest case that $A(x,t)$ is a solution to the Yang-Mills flow over $\mathbb{R}^n$, one can define
\begin{equation}\label{monotonicity}
\Phi_{x_0,t_0}(A(x,t))=(t_0-t)^2\int_{\mathbb{R}^n}|F(x,t)|^2G_{x_0,t_0}(x,t)dx,
\end{equation}
here $t_0>0, t\in [0,\min\{T,t_0\})$, and
$G_{x_0,t_0}(x,t)=[4\pi (t_0-t)]^{-\frac{n}{2}}\exp(-\frac{|x-x_0|^2}{4(t_0-t)})$ is the backward heat kernel.
The monotonicity formula of the Yang-Mills flow reads
\begin{equation}\label{monotonicityformula}
\frac{d}{dt}\Phi_{x_0,t_0}(A(x,t))=-2(t_0-t)^2\int_{\mathbb{R}^n}|\nabla^pF_{pj}-\frac{1}{2(t_0-t)}(x-x_0)^pF_{pj}|^2G_{x_0,t_0}(x,t) d x.
\end{equation}
The monotonicity $\Phi_{x_0,t_0}$ is non-increasing in $t$, and is preserved if and only if
\begin{equation}\label{shrinker1}
\nabla^pF_{pj}-\frac{1}{2(t_0-t)}(x-x_0)^pF_{pj}=0.
\end{equation}

For the limiting Yang-Mills flow $\widetilde{A}(y,s)$ obtained in Theorem \ref{Weinkove}
and any $(x_0,t_0)\in \mathbb{R}^n\times (0,+\infty)$, one can translate it into
\begin{equation}\label{firsttypesoliton}
A(x,t)=A_p(x,t)dx^p=\widetilde{A}_p(x-x_0,t-t_0)dx^p,
\end{equation}
then $A(x,t)$ is a solution to the Yang-Mills flow and (\ref{shrinker1}) is satisfied.
On the other hand if a connection $A(x)$ on a trivial $G$-vector bundle over $\mathbb{R}^n$ satisfying
$$\nabla^pF_{pj}-\frac{1}{2t_0}(x-x_0)^pF_{pj}=0,$$
then, in the exponential gauge for $A(x)$, i.e. a gauge such that $(x-x_0)^pA_p(x)=0$, the flow of connections given by
$$A_p(x,t):=\sqrt{\frac{t_0}{t_0-t}}A_p(x_0+\sqrt{\frac{t_0}{t_0-t}}(x-x_0))$$
is a solution to the Yang-Mills flow which satisfies $(\ref{shrinker1})$.
All these amount to say that limiting flows $\widetilde{A}(y,s)$, homothetically shrinking solitons $A(x)$ and homothetically
shrinking Yang-Mills flows are the same thing.

From now on we assume that $E$ is a trivial $G$-vector bundle over $\mathbb{R}^n$.

\begin{definition}
A connection $A(x)$ on $E$ is called
a homothetically shrinking soliton centered at $(x_0,t_0)$ if it satisfies
\begin{equation}\label{solitondef}
\nabla^pF_{pj}-\frac{1}{2t_0}(x-x_0)^pF_{pj}=0.
\end{equation}
\end{definition}
Let $A(x)$ be a homothetically shrinking soliton centered at $(x_0,t_0)$ and $A(x,t)$ the Yang-Mills flow initiating from $A(x)$.
In an exponential gauge such that $(x-x_0)^pA_p(x,t)=0$, we have for any $\lambda>0$ and any $t<t_0$ that
$A_j(x,t)=\lambda A_j(\lambda(x-x_0)+x_0,\lambda^2(t-t_0)+t_0)$.

\section{$\mathcal{F}$-functional and its first variation}

In this section we define the $\mathcal{F}$-functional of connections on the trivial $G$-vector bundle $E$ over $\mathbb{R}^n$.
Homothetically shrinking solitons are critical points of the $\mathcal{F}$-functional.
We shall prove necessary integral identities for homothetically shrinking solitons.
As a corollary of one of these identities, we give a proof of the fact that
the Yang-Mills flow in dimension four cannot develop a Type-I singularity.

For convenience, we set two $\mathfrak{g}$-valued $1$-forms $J$ and $X$, respectively, by
\begin{equation*}\label{JK}
J:=\nabla^pF_{pj}dx^j, \quad X:=i_{x-x_0}F=(x-x_0)^p F_{pj}dx^j.
\end{equation*}
According to (\ref{solitondef}), $A(x)$ is a homothetically shrinking soliton centered at $(x_0,t_0)$ if and only if
$$J=\frac{1}{2t_0}X.$$
We also set
\begin{equation}\label{setS}
\mathcal{S}_{x_0,t_0}=\{A(x): A \, \text{is a homothetically shrinking soliton with}\, \sup|\nabla^kA|<\infty, \forall k\geq 1\}.
\end{equation}
Note that for any $k\geq 1$, any time-slice $\widetilde{A}(\cdot,s)$ in Theorem \ref{Weinkove} satisfies
$\sup|\nabla^k\widetilde{A}(\cdot,s)|<\infty$.

\begin{definition}
For any $x_0\in \mathbb{R}^n, t_0>0$, the $\mathcal{F}$-functional with respect to $(x_0,t_0)$ is defined by
\begin{equation}\label{Ffunctional}
\mathcal{F}_{x_0,t_0}(A)=t_0^2\int_{\mathbb{R}^n}|F|^2(4\pi t_0)^{-\frac{n}{2}}e^{-\frac{|x-x_0|^2}{4t_0}}dx.
\end{equation}
\end{definition}

We now compute the first variation of the $\mathcal{F}$-functional.
Consider a differentiable $1$-parameter family $(x_s,t_s,A_s)$, where $A_0=A$.
Denote
$$ \dot{t}_s=\frac{d}{ds}t_s, \quad  \dot{x}_s=\frac{d}{ds}x_s,\quad  \theta_s=\frac{d}{ds}A_s,$$
and
$$G_s(x)=(4\pi t_s)^{-\frac{n}{2}}e^{-\frac{|x-x_s|^2}{4t_s}}.$$

\begin{proposition}\label{firstvariationprop}
Assume $|\nabla^kA_s|<\infty$ for any $k\geq 1$ and
$\int_{\mathbb{R}^n}(|\theta_s|^2+|\nabla \theta_s|^2)G_sdx<\infty.$
The first variation of the $\mathcal{F}$-functional is given by
\begin{eqnarray}
\frac{d}{d s}\mathcal{F}_{x_s,t_s}(A_s)&=&
\int_{\mathbb{R}^n}\dot{t}_s(\frac{4-n}{2}t_s+\frac{1}{4}|x-x_s|^2)|F_s|^2G_s(x)dx\nonumber
\\&&+\int_{\mathbb{R}^n}\frac{1}{2}t_s<\dot{x}_s,x-x_s>|F_s|^2G_s(x)dx\nonumber
\\&&-\int_{\mathbb{R}^n}2t_s^2<\theta_s,J_s-\frac{X_s}{2t_s}>G_s(x)dx.\label{firstvariation}
\end{eqnarray}
\end{proposition}

\begin{proof}
Note that
$$\frac{\partial}{\partial s}G_s(x)=(-\frac{n}{2}\frac{\dot{t}_s}{t_s}+\frac{\dot{t}_s|x-x_s|^2}{4t_s^2}+\frac{<\dot{x}_s,x-x_s>}{2t_s})G_s(x),$$
and
$$\frac{\partial}{\partial s}|F_s|^2=F_{ij\beta}^\alpha(\nabla_i\theta_{j\beta}^\alpha-\nabla_j\theta_{i\beta}^\alpha),$$
so we have
\begin{eqnarray*}
\frac{d}{d s}\mathcal{F}_{x_s,t_s}(A_s)&=&\int_{\mathbb{R}^n}2t_s\dot{t}_s|F_s|^2G_s(x)dx
+\int_{\mathbb{R}^n}t_s^2F_{ij\beta}^\alpha(\nabla_i\theta_{j\beta}^\alpha-\nabla_j\theta_{i\beta}^\alpha)G_s(x)dx
\\&&+\int_{\mathbb{R}^n}t_s^2|F_s|^2(-\frac{n}{2}\frac{\dot{t}_s}{t_s}+\frac{\dot{t}_s|x-x_s|^2}{4t_s^2}+\frac{<\dot{x}_s,x-x_s>}{2t_s})G_s(x)dx
\\&=&\int_{\mathbb{R}^n}2t_s\dot{t}_s|F_s|^2G_s(x)dx
+\int_{\mathbb{R}^n}2t_s^2F_{ij\beta}^\alpha\nabla_i\theta_{j\beta}^\alpha G_s(x)dx
\\&&+\int_{\mathbb{R}^n}t_s^2|F_s|^2(-\frac{n}{2}\frac{\dot{t}_s}{t_s}+\frac{\dot{t}_s|x-x_s|^2}{4t_s^2}+\frac{<\dot{x}_s,x-x_s>}{2t_s})G_s(x)dx.
\end{eqnarray*}
Let $\eta(x)$ be a cutoff function on $\mathbb{R}^n$.
By integration by parts, we have
\begin{eqnarray}
&&\int_{\mathbb{R}^n}2t_s^2F_{ij\beta}^\alpha\nabla_i\theta_{j\beta}^\alpha G_s(x)\eta(x)dx\nonumber
\\&=&\int_{\mathbb{R}^n}-2t_s^2\theta_{j\beta}^\alpha[\nabla_iF_{ij\beta}^\alpha G_s\eta+F_{ij\beta}^\alpha \partial_i(G_s)\eta
+F_{ij\beta}^\alpha G_s\partial_i\eta]dx\nonumber
\\&=&\int_{\mathbb{R}^n}-2t_s^2\theta_{j\beta}^\alpha[\nabla_iF_{ij\beta}^\alpha \eta-\frac{(x-x_s)^i}{2t_s}F_{ij\beta}^\alpha \eta
+F_{ij\beta}^\alpha \partial_i\eta]G_sdx.\label{cutoff}
\end{eqnarray}
Let $\eta_l (x) = 1$ for $|x|\leq l$, and cut off to zero linearly on $B_{l+1}\setminus B_l$.
Taking $\eta=\eta_l$ in (\ref{cutoff}) and applying the Lebesgue's dominated convergence theorem, we get
\begin{equation}\label{integrationbyparts1}
\int_{\mathbb{R}^n}2t_s^2F_{ij\beta}^\alpha\nabla_i\theta_{j\beta}^\alpha G_s(x)dx
=\int_{\mathbb{R}^n}\theta_{j\beta}^\alpha[-2t_s^2\nabla_iF_{ij\beta}^\alpha +t_s(x-x_s)^iF_{ij\beta}^\alpha ]G_sdx.
\end{equation}
Hence we get
\begin{eqnarray*}
\frac{d}{d s}\mathcal{F}_{x_s,t_s}(A_s)&=&\int_{\mathbb{R}^n}2t_s\dot{t}_s|F_s|^2G_s(x)dx
\\&&+\int_{\mathbb{R}^n}\theta_{j\beta}^\alpha[-2t_s^2\nabla_iF_{ij\beta}^\alpha +t_s(x-x_s)^iF_{ij\beta}^\alpha]G_s(x)dx
\\&&+\int_{\mathbb{R}^n}t_s^2|F_s|^2(-\frac{n}{2}\frac{\dot{t}_s}{t_s}+\frac{\dot{t}_s|x-x_s|^2}{4t_s^2}+\frac{<\dot{x}_s,x-x_s>}{2t_s})G_s(x)dx.
\\&=&\int_{\mathbb{R}^n}\dot{t}_s(\frac{4-n}{2}t_s+\frac{1}{4}|x-x_s|^2)|F_s|^2G_s(x)dx
\\&&+\int_{\mathbb{R}^n}\frac{1}{2}t_s<\dot{x}_s,x-x_s>|F_s|^2G_s(x)dx
\\&&-\int_{\mathbb{R}^n}2t_s^2<\theta_s,J_s-\frac{X_s}{2t_s}>G_s(x)dx.
\end{eqnarray*}

\end{proof}

From Proposition \ref{firstvariationprop}, we have the following

\begin{corollary}
A connection $A(x)$ is a critical point of $\mathcal{F}_{x_0,t_0}$
if and only if $A(x)$ is a homothetically shrinking soliton centered at $(x_0,t_0)$.
\end{corollary}

We shall check that $(A(x),x_0,t_0)$ is a critical point of the $\mathcal{F}$-functional $(\widetilde{A},x,t)\mapsto F_{x,t}(\widetilde{A})$
if and only if $A(x)$ is a homothetically shrinking soliton centered at $(x_0,t_0)$.
To check this we need some identities for homothetically shrinking solitons.
We also need such identities in the calculation of the second variation of the $\mathcal{F}$-functional in the next section.
Denote
$$G(x)=(4\pi t_0)^{-\frac{n}{2}}e^{-\frac{|x-x_0|^2}{4t_0}}.$$

\begin{lemma}\label{identities}
Let $A(x)$ be a homothetically shrinking soliton centered at $(x_0,t_0)$ and $\sup|F(x)|<\infty$.
Let $\varphi=\varphi^p\partial_p$ be a vector field on $\mathbb{R}^n$ such that $|\varphi|$ is a polynomial in $|x-x_0|$, and
$V$ a vector in $\mathbb{R}^n$. Then we have
\begin{equation*}
\int_{\mathbb{R}^n}\varphi^p(x-x_0)^p|F|^2G(x)dx
=\int_{\mathbb{R}^n}[2t_0\partial_p(\varphi^p)|F|^2-4t_0\partial_i\varphi^pF_{pj\beta}^\alpha F_{ij\beta}^\alpha]G(x) dx.
\end{equation*}
In particular,
\begin{itemize}
\item [(a)]
$\int_{\mathbb{R}^n}|x-x_0|^2|F|^2G(x)dx=\int_{\mathbb{R}^n}2(n-4)t_0|F|^2G(x) dx;$
\item [(b)]
$\int_{\mathbb{R}^n}(x-x_0)^k|F|^2G(x)dx=0;$
\item [(c)]
$\int_{\mathbb{R}^n}|x-x_0|^4|F|^2G(x)dx=\int_{\mathbb{R}^n}[4(n-2)(n-4)t_0^2|F|^2-32t_0^3|J|^2]G dx;$
\item [(d)]
$\int_{\mathbb{R}^n}|x-x_0|^2<V,x-x_0>|F|^2G(x)dx=0;$
\item [(e)]
$\int_{\mathbb{R}^n}<x-x_0,V>^2|F|^2Gdx=\int_{\mathbb{R}^n}(2t_0|V|^2|F|^2-4t_0<V^iF_{ij},V^pF_{pj}>)G dx$.
\end{itemize}
\end{lemma}

\begin{proof}
Let $\eta(x)$ be a cutoff function on $\mathbb{R}^n$. By integration by parts, we get
\begin{eqnarray*}
&&\int_{\mathbb{R}^n}\varphi^p(x-x_0)^p|F|^2G(x)\eta(x)dx
\\&=&\int_{\mathbb{R}^n}-2t_0\varphi^p|F|^2\partial_p G(x)\eta(x)dx
\\&=&\int_{\mathbb{R}^n}2t_0[\partial_p(\varphi^p)|F|^2\eta+\varphi^p\partial_p(|F|^2)\eta+\varphi^p|F|^2\partial_p\eta]G(x)dx.
\end{eqnarray*}
By integration by parts we have
\begin{eqnarray*}
\int_{\mathbb{R}^n}4t_0\varphi^pF_{pj\beta}^\alpha J_{j\beta}^\alpha G\eta dx
&=&\int_{\mathbb{R}^n}4t_0\varphi^p[\nabla_i(F_{pj\beta}^\alpha F_{ij\beta}^\alpha)-\nabla_iF_{pj\beta}^\alpha F_{ij\beta}^\alpha] G\eta dx
\\&=&\int_{\mathbb{R}^n}-4t_0F_{pj\beta}^\alpha F_{ij\beta}^\alpha[\partial_i\varphi^p-\frac{(x-x_0)^i}{2t_0}\varphi^p]G\eta dx
\\&&-\int_{\mathbb{R}^n}2t_0\varphi^p(\nabla_iF_{pj\beta}^\alpha F_{ij\beta}^\alpha+\nabla_jF_{ip\beta}^\alpha F_{ij\beta}^\alpha) G\eta dx
\\&&-\int_{\mathbb{R}^n}4t_0\varphi^pF_{pj\beta}^\alpha F_{ij\beta}^\alpha G\partial_i\eta dx.
\end{eqnarray*}
It then follows from the Bianchi identity that
\begin{eqnarray*}
\int_{\mathbb{R}^n}4t_0\varphi^pF_{pj\beta}^\alpha J_{j\beta}^\alpha G\eta dx
&=&\int_{\mathbb{R}^n}-4t_0F_{pj\beta}^\alpha F_{ij\beta}^\alpha[\partial_i\varphi^p-\frac{(x-x_0)^i}{2t_0}\varphi^p]G\eta dx
\\&&-\int_{\mathbb{R}^n}2t_0\varphi^p\nabla_pF_{ij\beta}^\alpha F_{ij\beta}^\alpha G\eta dx
-\int_{\mathbb{R}^n}4t_0\varphi^pF_{pj\beta}^\alpha F_{ij\beta}^\alpha G\partial_i\eta dx
\\&=&\int_{\mathbb{R}^n}-4t_0F_{pj\beta}^\alpha F_{ij\beta}^\alpha[\partial_i\varphi^p-\frac{(x-x_0)^i}{2t_0}\varphi^p]G\eta dx
\\&&-\int_{\mathbb{R}^n}2t_0\varphi^p\partial_p(|F|^2) G\eta dx
-\int_{\mathbb{R}^n}4t_0\varphi^pF_{pj\beta}^\alpha F_{ij\beta}^\alpha G\partial_i\eta dx,
\end{eqnarray*}
i.e.
\begin{eqnarray*}
\int_{\mathbb{R}^n}2t_0\varphi^p\partial_p(|F|^2) G\eta dx
&=&-\int_{\mathbb{R}^n}4t_0\varphi^pF_{pj\beta}^\alpha J_{j\beta}^\alpha G\eta dx
\\&&-\int_{\mathbb{R}^n}4t_0F_{pj\beta}^\alpha F_{ij\beta}^\alpha[\partial_i\varphi^p-\frac{(x-x_0)^i}{2t_0}\varphi^p]G\eta dx
\\&&-\int_{\mathbb{R}^n}4t_0\varphi^pF_{pj\beta}^\alpha F_{ij\beta}^\alpha G\partial_i\eta dx.
\end{eqnarray*}
Thus we have
\begin{eqnarray*}
&&\int_{\mathbb{R}^n}\varphi^p(x-x_0)^p|F|^2G(x)\eta(x)dx
\\&=&\int_{\mathbb{R}^n}2t_0[\partial_p(\varphi^p)|F|^2\eta+\varphi^p\partial_p(|F|^2)\eta+\varphi^p|F|^2\partial_p\eta]G(x)dx
\\&=&\int_{\mathbb{R}^n}2t_0\partial_p(\varphi^p)|F|^2G\eta dx
-\int_{\mathbb{R}^n}4t_0\varphi^pF_{pj\beta}^\alpha J_{j\beta}^\alpha G\eta dx
\\&&-\int_{\mathbb{R}^n}4t_0F_{pj\beta}^\alpha F_{ij\beta}^\alpha[\partial_i\varphi^p-\frac{(x-x_0)^i}{2t_0}\varphi^p]G\eta dx
\\&&-\int_{\mathbb{R}^n}4t_0\varphi^pF_{pj\beta}^\alpha F_{ij\beta}^\alpha G\partial_i\eta dx+\int_{\mathbb{R}^n}2t_0\varphi^p|F|^2G\partial_p\eta dx
\\&=&\int_{\mathbb{R}^n}[2t_0\partial_p(\varphi^p)|F|^2-4t_0\partial_i\varphi^pF_{pj\beta}^\alpha F_{ij\beta}^\alpha]G\eta dx
\\&&-\int_{\mathbb{R}^n}4t_0\varphi^pF_{pj\beta}^\alpha (J_{j\beta}^\alpha -\frac{1}{2t_0}X_{j\beta}^\alpha)G\eta dx
\\&&-\int_{\mathbb{R}^n}4t_0\varphi^pF_{pj\beta}^\alpha F_{ij\beta}^\alpha G\partial_i\eta dx+\int_{\mathbb{R}^n}2t_0\varphi^p|F|^2G\partial_p\eta dx.
\end{eqnarray*}
Therefore for a homothetically shrinking soliton centered at $(x_0,t_0)$,
\begin{eqnarray}\label{identityformula0}
\int_{\mathbb{R}^n}\varphi^p(x-x_0)^p|F|^2G\eta dx
&=&\int_{\mathbb{R}^n}[2t_0\partial_p(\varphi^p)|F|^2-4t_0\partial_i\varphi^pF_{pj\beta}^\alpha F_{ij\beta}^\alpha]G\eta dx\nonumber
\\&&-\int_{\mathbb{R}^n}4t_0\varphi^pF_{pj\beta}^\alpha F_{ij\beta}^\alpha G\partial_i\eta dx+\int_{\mathbb{R}^n}2t_0\varphi^p|F|^2G\partial_p\eta dx.
\end{eqnarray}
Applying to (\ref{identityformula0}) with $\eta(x)=\eta_l (x)$,
where $\eta_l (x)=1$ for $|x|\leq l$ and is cut off to zero linearly on $B_{l+1}\setminus B_l$, we get
\begin{equation}\label{identityformula}
\int_{\mathbb{R}^n}\varphi^p(x-x_0)^p|F|^2G dx
=\int_{\mathbb{R}^n}[2t_0\partial_p(\varphi^p)|F|^2-4t_0\partial_i\varphi^pF_{pj\beta}^\alpha F_{ij\beta}^\alpha]G dx.
\end{equation}

Taking $\varphi^p=(x-x_0)^p$, by (\ref{identityformula}) we get
\begin{eqnarray*}
\int_{\mathbb{R}^n}|x-x_0|^2|F|^2G(x)dx=\int_{\mathbb{R}^n}2(n-4)t_0|F|^2G(x) dx.
\end{eqnarray*}

Taking $\varphi^p=\delta_k^p$, by (\ref{identityformula}) we get for any $k=1,\cdots,n$,
$$\int_{\mathbb{R}^n}(x-x_0)^k|F|^2G(x)dx=0.$$

Taking $\varphi^p=|x-x_0|^2(x-x_0)^p$, by (\ref{identityformula}) and (a) we get
\begin{eqnarray*}
&&\int_{\mathbb{R}^n}|x-x_0|^4|F|^2G(x)dx
\\&=&\int_{\mathbb{R}^n}[2t_0(n+2)|x-x_0|^2|F|^2-8t_0|x-x_0|^2|F|^2-8t_0|X|^2]G dx
\\&=&\int_{\mathbb{R}^n}[4(n-2)(n-4)t_0^2|F|^2-32t_0^3|J|^2]G dx.
\end{eqnarray*}

Taking $\varphi^p=|x-x_0|^2V^p$, by (\ref{identityformula}) and (b) we get
\begin{eqnarray*}
\int_{\mathbb{R}^n}|x-x_0|^2<V,x-x_0>|F|^2G(x)dx
&=&\int_{\mathbb{R}^n}-16t_0^2<J_j,V^pF_{pj}>G dx.
\end{eqnarray*}
On the other hand taking $\varphi^p=<V,x-x_0>(x-x_0)^p$, by (\ref{identityformula}) and (b) we get
\begin{eqnarray*}
\int_{\mathbb{R}^n}|x-x_0|^2<V,x-x_0>|F|^2G(x)dx
&=&\int_{\mathbb{R}^n}-8t_0^2<J_j,V^iF_{ij}>G dx.
\end{eqnarray*}
Thus we have
$$\int_{\mathbb{R}^n}|x-x_0|^2<V,x-x_0>|F|^2G(x)dx=\int_{\mathbb{R}^n}<J_j,V^pF_{pj}>G dx=0.$$

Taking $\varphi^p=<V,x-x_0>V^p$, by (\ref{identityformula}) we get
\begin{eqnarray*}
\int_{\mathbb{R}^n}<x-x_0,V>^2|F|^2Gdx=\int_{\mathbb{R}^n}(2t_0|V|^2|F|^2-4t_0<V^iF_{ij},V^pF_{pj}>)G dx.
\end{eqnarray*}
\end{proof}

By the first variation formula (\ref{firstvariation}), (a) and (b) of Lemma \ref{identities} we get the following

\begin{corollary}\label{criticalofentropy}
$(A(x),x_0,t_0)$ is a critical point of the $\mathcal{F}$-functional if and only if
$A(x)$ is a homothetically shrinking soliton centered at $(x_0,t_0)$.
\end{corollary}

\begin{corollary}\label{notypeI}
When $n=2, 3$, or $4$, there exists no homothetically shrinking soliton
such that $|F|$ is uniformly bounded and not identically zero.
In particular in dimension four, the Yang-Mills flow on $E_M$ cannot develop a singularity of Type-I.
\end{corollary}
\begin{proof}
The first part follows from Lemma \ref{identities} (a).
By Weinkove's result \cite{Weinkove}, see also Section 2, at a Type-I singularity of a Yang-Mills flow
one can obtain a homothetically shrinking soliton on a trivial $G$-vector bundle over $\mathbb{R}^n$
whose curvature is uniformly bounded and non-zero. Therefore in dimension four if a Type-I singularity occurs,
it would contradict with the non-existence of such a homothetically shrinking soliton.
\end{proof}

\section{Second variation of $\mathcal{F}$-functional}

We now compute the second variation of the $\mathcal{F}$-functional at a homothetically shrinking soliton $A(x)$.
Let $d^\nabla$ denote the covariant exterior differentiation on $\mathfrak{g}$-valued forms
and $(d^\nabla)^*$ denote the formal adjoint of $d^\nabla$. For a $\mathfrak{g}$-valued $1$-form $\theta$, let
\begin{equation}\label{R}
\mathcal{R}(\theta_j)=\mathcal{R}(\theta)(\partial_j):=[F_{ij},\theta_i],
\end{equation}
and
\begin{equation}\label{Jacobioperator}
L\theta:=-t_0[(d^\nabla)^*d^\nabla\theta+\mathcal{R}(\theta)+i_{\frac{1}{2t_0}(x-x_0)}d^\nabla\theta].
\end{equation}
We also introduce the space
\begin{equation}
W_G^{2,2}:=\{\theta: \int_{\mathbb{R}^n}(|\theta|^2+|\nabla \theta|^2+|L\theta|^2)G(x)dx<\infty\}.
\end{equation}
Denote
$$\dot{t}_s|_{s=0}=q,\quad  \dot{x}_s|_{s=0}=V, \quad  \theta=\frac{d}{ds}|_{s=0}A_s,$$
$$\mathcal{F}_{x_0,t_0}''(q,V,\theta)=\frac{d^2}{d s^2}|_{s=0}\mathcal{F}_{x_s,t_s}(A_s).$$

\begin{proposition}\label{secondvariation0}
Let $A(x)$ be a homothetically shrinking soliton in $\mathcal{S}_{x_0,t_0}$, see (\ref{setS}).
Then for any $\theta\in W_G^{2,2}$, we have
\begin{equation}\label{secondvariation}
\frac{1}{2t_0}\mathcal{F}_{x_0,t_0}''(q,V,\theta)=\int_{\mathbb{R}^n}<-L\theta-2qJ-i_VF, \theta>G dx
-\int_{\mathbb{R}^n}(q^2|J|^2+\frac{1}{2}|i_V F|^2)G dx.
\end{equation}
\end{proposition}

\begin{proof}
Recall that
\begin{eqnarray*}
\frac{d}{d s}\mathcal{F}_{x_s,t_s}(A_s)&=&
\int_{\mathbb{R}^n}\dot{t}_s(\frac{4-n}{2}t_s+\frac{1}{4}|x-x_s|^2)|F_s|^2G_s(x)dx
\\&&+\int_{\mathbb{R}^n}\frac{1}{2}t_s<\dot{x}_s,x-x_s>|F_s|^2G_s(x)dx
\\&&-\int_{\mathbb{R}^n}2t_s^2<J_s-\frac{X_s}{2t_s},\theta_s>G_s(x)dx.
\end{eqnarray*}
By the assumption that $A(x)\in \mathcal{S}_{x_0,t_0}$ and Lemma \ref{identities} (a, b), we have
\begin{eqnarray*}
\mathcal{F}_{x_0,t_0}''(q,V,\theta)&=&
\int_{\mathbb{R}^n}[q(\frac{4-n}{2}q-\frac{1}{2}<x-x_0,V>)+\frac{1}{2}t_0<V,-V>]|F|^2G dx
\\&&+\int_{\mathbb{R}^n}[q(\frac{4-n}{2}t_0+\frac{1}{4}|x-x_0|^2)+\frac{1}{2}t_0<V,x-x_0>]\frac{\partial |F_s|^2}{\partial s}|_{s=0}G dx
\\&&+\int_{\mathbb{R}^n}[q(\frac{4-n}{2}t_0+\frac{1}{4}|x-x_0|^2)+\frac{1}{2}t_0<V,x-x_0>]|F|^2\frac{\partial G_s}{\partial s}|_{s=0}dx
\\&&-\int_{\mathbb{R}^n}2t_0^2<\frac{\partial}{\partial s}|_{s=0}(J_s-\frac{X_s}{2t_s}),\theta>G dx.
\end{eqnarray*}
Note that
\begin{eqnarray*}
\frac{\partial |F_s|^2}{\partial s}|_{s=0}&=&F_{ij\beta}^\alpha(\nabla_i\theta_{j\beta}^\alpha-\nabla_j\theta_{i\beta}^\alpha)
=2F_{ij\beta}^\alpha \nabla_i\theta_{j\beta}^\alpha,
\end{eqnarray*}
$$\frac{\partial G_s}{\partial s}|_{s=0}=(-\frac{n}{2}\frac{q}{t_0}+\frac{q|x-x_0|^2}{4t_0^2}+\frac{<V,x-x_0>}{2t_0})G(x),$$
\begin{eqnarray*}
\frac{\partial}{\partial s}|_{s=0}J_{j\beta}^\alpha
=\nabla_p\nabla_p\theta_{j\beta}^\alpha-\nabla_p\nabla_j\theta_{p\beta}^\alpha
+\theta_{p\gamma}^\alpha F_{pj\beta}^\gamma-F_{pj\gamma}^\alpha \theta_{p\beta}^\gamma,
\end{eqnarray*}
\begin{eqnarray*}
\frac{\partial}{\partial s}|_{s=0}(-\frac{1}{2t_s}X_{j\beta}^\alpha)
&=&\frac{q}{2t_0^2}X_{j\beta}^\alpha+\frac{1}{2t_0}V^kF_{kj\beta}^\alpha
-\frac{1}{2t_0}(x-x_0)^k(\nabla_k\theta_{j\beta}^\alpha-\nabla_j\theta_{k\beta}^\alpha).
\end{eqnarray*}
Thus we get
\begin{eqnarray*}
\mathcal{F}_{x_0,t_0}''(q,V,\theta)&=&
\int_{\mathbb{R}^n}[q(\frac{4-n}{2}q-\frac{1}{2}<x-x_0,V>)-\frac{1}{2}t_0|V|^2]|F|^2G dx
\\&&+\int_{\mathbb{R}^n}[q(\frac{4-n}{2}t_0+\frac{1}{4}|x-x_0|^2)+\frac{1}{2}t_0<V,x-x_0>]2F_{ij\beta}^\alpha \nabla_i\theta_{j\beta}^\alpha G dx
\\&&+\int_{\mathbb{R}^n}[q(\frac{4-n}{2}t_0+\frac{1}{4}|x-x_0|^2)+\frac{1}{2}t_0<V,x-x_0>]|F|^2
\\&&\quad \times (-\frac{n}{2}\frac{q}{t_0}+\frac{q|x-x_0|^2}{4t_0^2}+\frac{<V,x-x_0>}{2t_0})G dx
\\&&-\int_{\mathbb{R}^n}2t_0^2[\nabla_p(\nabla_p\theta_{j\beta}^\alpha-\nabla_j\theta_{p\beta}^\alpha)
+\theta_{p\gamma}^\alpha F_{pj\beta}^\gamma-F_{pj\gamma}^\alpha \theta_{p\beta}^\gamma]\theta_{j\beta}^\alpha G dx
\\&&-\int_{\mathbb{R}^n}2t_0^2[\frac{q}{2t_0^2}X_{j\beta}^\alpha+\frac{1}{2t_0}V^kF_{kj\beta}^\alpha
-\frac{1}{2t_0}(x-x_0)^k(\nabla_k\theta_{j\beta}^\alpha-\nabla_j\theta_{k\beta}^\alpha)]\theta_{j\beta}^\alpha G dx.
\end{eqnarray*}
By integration by parts, we have
\begin{eqnarray*}
&&\int_{\mathbb{R}^n}[q(\frac{4-n}{2}t_0+\frac{1}{4}|x-x_0|^2)+\frac{1}{2}t_0<V,x-x_0>]2F_{ij\beta}^\alpha \nabla_i\theta_{j\beta}^\alpha G dx
\\&=&\int_{\mathbb{R}^n}-2[q(\frac{4-n}{2}t_0+\frac{1}{4}|x-x_0|^2)+\frac{1}{2}t_0<V,x-x_0>]<J-\frac{1}{2t_0}X,\theta> G dx
\\&&-\int_{\mathbb{R}^n}2[\frac{1}{2}q(x-x_0)^i+\frac{1}{2}t_0V^i]F_{ij\beta}^\alpha \theta_{j\beta}^\alpha G dx
\\&=&\int_{\mathbb{R}^n}[-q(x-x_0)^i-t_0V^i]F_{ij\beta}^\alpha \theta_{j\beta}^\alpha G dx.
\end{eqnarray*}
Then by using Lemma \ref{identities}, we have
\begin{eqnarray*}
\mathcal{F}_{x_0,t_0}''(q,V,\theta)&=&
\int_{\mathbb{R}^n}[\frac{4-n}{2}q^2-\frac{1}{2}t_0|V|^2]|F|^2G dx
\\&&+\int_{\mathbb{R}^n}[-q(x-x_0)^i-t_0V^i]F_{ij\beta}^\alpha \theta_{j\beta}^\alpha G dx
\\&&+\int_{\mathbb{R}^n}[\frac{n-4}{2}q^2|F|^2-2t_0q^2|J|^2]Gdx
\\&&+\int_{\mathbb{R}^n}\frac{1}{4}(2t_0|V|^2|F|^2-4t_0<V^iF_{ij},V^pF_{pj}>)G dx
\\&&-\int_{\mathbb{R}^n}2t_0^2[\nabla_p(\nabla_p\theta_{j\beta}^\alpha-\nabla_j\theta_{p\beta}^\alpha)
+\theta_{p\gamma}^\alpha F_{pj\beta}^\gamma-F_{pj\gamma}^\alpha \theta_{p\beta}^\gamma]\theta_{j\beta}^\alpha G dx
\\&&-\int_{\mathbb{R}^n}[q(x-x_0)^i+t_0V^i]F_{ij\beta}^\alpha \theta_{j\beta}^\alpha G dx
\\&&+\int_{\mathbb{R}^n}t_0(x-x_0)^k(\nabla_k\theta_{j\beta}^\alpha-\nabla_j\theta_{k\beta}^\alpha)\theta_{j\beta}^\alpha G dx.
\end{eqnarray*}
Thus,
\begin{eqnarray*}
\mathcal{F}_{x_0,t_0}''(q,V,\theta)&=&
\int_{\mathbb{R}^n}[-2q(x-x_0)^i-2t_0V^i]F_{ij\beta}^\alpha \theta_{j\beta}^\alpha G dx
\\&&-\int_{\mathbb{R}^n}2t_0q^2|J|^2Gdx
-\int_{\mathbb{R}^n}t_0<V^iF_{ij},V^pF_{pj}>)G dx
\\&&-\int_{\mathbb{R}^n}2t_0^2[\nabla_p(\nabla_p\theta_{j\beta}^\alpha-\nabla_j\theta_{p\beta}^\alpha)
+\theta_{p\gamma}^\alpha F_{pj\beta}^\gamma-F_{pj\gamma}^\alpha \theta_{p\beta}^\gamma]\theta_{j\beta}^\alpha G dx
\\&&+\int_{\mathbb{R}^n}t_0(x-x_0)^k(\nabla_k\theta_{j\beta}^\alpha-\nabla_j\theta_{k\beta}^\alpha)\theta_{j\beta}^\alpha G dx.
\end{eqnarray*}
Note that
$$(d^\nabla)^*d^\nabla\theta_j=-\nabla_p(\nabla_p\theta_{j}-\nabla_j\theta_{p}),$$
$$\mathcal{R}(\theta_j)=[F_{pj},\theta_p]=F_{pj}\theta_p-\theta_pF_{pj},$$
$$i_{x-x_0}d^\nabla\theta_j=(x-x_0)^k(\nabla_k\theta_{j}-\nabla_j\theta_{k}),$$
so we have
\begin{eqnarray*}
\mathcal{F}_{x_0,t_0}''(q,V,\theta)&=&
\int_{\mathbb{R}^n}2t_0^2<(d^\nabla)^*d^\nabla\theta_j+\mathcal{R}(\theta_j)+i_{\frac{1}{2t_0}(x-x_0)}d^\nabla\theta_j,\theta_j>G dx
\\&&-\int_{\mathbb{R}^n}2t_0<2qJ_j+V^iF_{ij}, \theta_j>G dx
\\&&-2t_0\int_{\mathbb{R}^n}(q^2|J|^2+\frac{1}{2}|i_VF|^2)G dx.
\end{eqnarray*}
Let
$$L=-t_0[(d^\nabla)^*d^\nabla +\mathcal{R}+i_{\frac{1}{2t_0}(x-x_0)}d^\nabla],$$
then we have
\begin{eqnarray*}
\frac{1}{2t_0}\mathcal{F}_{x_0,t_0}''(q,V,\theta)&=&
\int_{\mathbb{R}^n}<-L\theta-2qJ-i_VF, \theta>G dx
-\int_{\mathbb{R}^n}(q^2|J|^2+\frac{1}{2}|i_VF|^2)G dx.
\end{eqnarray*}
\end{proof}

\section{F-stability and its characterization}

In this section we define the F-stability for homothetically shrinking solitons in $\mathcal{S}_{x_0,t_0}$.
The operator $L$ admits eigenfields $J$ and $i_VF$ of eigenvalues $-1$ and $-\frac{1}{2}$, respectively.
F-stability is equivalent to the semi-positiveness of $L$ modulo the vector space spanned by $J$ and $i_VF$.
Let $C_0^\infty(\Omega^1\otimes \mathfrak{g})$, or simply $C_0^\infty$, denote the space of
$\mathfrak{g}$-valued $1$-forms with compact supports on $\mathbb{R}^n$.
The space $C_0^\infty$ is dense in $W_G^{2,2}$.

\begin{definition}
A homothetically shrinking soliton $A\in \mathcal{S}_{x_0,t_0}$ is called F-stable
if for any $\theta$ in $C_0^\infty$, or equivalently in $W_G^{2,2}$,
there exist a real number $q$ and a vector $V$ such that
$$\mathcal{F}_{x_0,t_0}''(q,V,\theta)\geq 0.$$
\end{definition}

Given a homothetically shrinking soliton $A\in \mathcal{S}_{x_0,t_0}$ with an exponential gauge, the rescaling
$$\widetilde{A}_i(x)=\sqrt{t_0}A_i(\sqrt{t_0}x+x_0)$$
is a homothetically shrinking soliton in $\mathcal{S}_{0,1}$.
Without loss of generality, in the remaining of this section we let $x_0=0$ and $t_0=1$.
Then
$$G(x)=(4\pi)^{-\frac{n}{2}}e^{-\frac{|x|^2}{4}}$$
and
$$L\theta=-[(d^\nabla)^*d^\nabla\theta+\mathcal{R}(\theta)+i_{\frac{x}{2}}d^\nabla\theta].$$
The operator $L$ is self-adjoint in the following sense: for any $\theta, \eta\in W_{G}^{2,2}$,
\begin{equation}\label{Ladjoint}
\int_{\mathbb{R}^n}<L\theta,\eta>Gdx=-\int_{\mathbb{R}^n}[<d^\nabla\theta,d^\nabla\eta>+<\mathcal{R}(\theta),\eta>]Gdx
=\int_{\mathbb{R}^n}<\theta,L\eta>Gdx.
\end{equation}
A $\mathfrak{g}$-valued $1$-form $\theta\in W_G^{2,2}$ is called an eigenfield of $L$ and of eigenvalue $\lambda$ if
$L\theta=-\lambda\theta$. We denote the eigenfield space of eigenvalue $\lambda$ by $E_\lambda$.

\begin{proposition}\label{twoeigenfunctions}
Let $A$ be a homothetically shrinking soliton in $\mathcal{S}_{0,1}$.  Then
\begin{equation}\label{J}
L J= J,
\end{equation}
and
\begin{equation}\label{iVF}
L (i_V F)= \frac{1}{2}i_V F, \quad \forall V\in \mathbb{R}^n.
\end{equation}
\end{proposition}
\begin{proof} Note that
$$J_{j}=\nabla_pF_{pj}=\frac{1}{2}x^pF_{pj},$$
$$L=-(d^\nabla)^*d^\nabla-\mathcal{R}-i_{\frac{x}{2}}d^\nabla,$$
and
\begin{eqnarray*}
LJ_j&=&\nabla_p\nabla_pJ_j-\nabla_p\nabla_jJ_p-[F_{pj},J_p]-\frac{1}{2}(d^\nabla J)(x^p\partial_p, \partial_j)
\\&=&\nabla_p\nabla_pJ_j-\nabla_p\nabla_jJ_p-[F_{pj},J_p]-\frac{1}{2}x^p(\nabla_pJ_j-\nabla_jJ_p).
\end{eqnarray*}
We have
$$\nabla_pJ_j=\nabla_p(\frac{1}{2}x^qF_{qj})=\frac{1}{2}F_{pj}+\frac{1}{2}x^q\nabla_pF_{qj},$$
then
\begin{eqnarray*}
\nabla_p\nabla_jJ_p&=&\nabla_p(-\frac{1}{2}F_{pj}+\frac{1}{2}x^q\nabla_jF_{qp})
=-\frac{1}{2}J_j-\frac{1}{2}x^q\nabla_p\nabla_jF_{pq},
\end{eqnarray*}
and by using the Bianchi identity and the Ricci formula, we get
\begin{eqnarray*}
\nabla_p\nabla_pJ_j&=&\nabla_pF_{pj}+\frac{1}{2}x^q\nabla_p\nabla_pF_{qj}
\\&=&\nabla_pF_{pj}+\frac{1}{2}x^q\nabla_p(-\nabla_qF_{jp}-\nabla_jF_{pq})
\\&=&\nabla_pF_{pj}-\frac{1}{2}x^q(\nabla_q\nabla_pF_{jp}+F_{pq}F_{jp}-F_{jp}F_{pq})-\frac{1}{2}x^q\nabla_p\nabla_jF_{pq}
\\&=&J_j+\frac{1}{2}x^q\nabla_qJ_j+[J_p, F_{jp}]-\frac{1}{2}x^q\nabla_p\nabla_jF_{pq}.
\end{eqnarray*}
Hence
$$LJ_j=\frac{3}{2}J_j+\frac{1}{2}x^p\nabla_jJ_p.$$
The identity (\ref{J}) then follows from
\begin{eqnarray*}
\frac{1}{2}x^p\nabla_jJ_p&=&\frac{1}{2}\nabla_j(x^pJ_p)-\frac{1}{2}J_j
=\frac{1}{2}\nabla_j(x^p\frac{1}{2}x^qF_{qp})-\frac{1}{2}J_j
=-\frac{1}{2}J_j.
\end{eqnarray*}

We now prove (\ref{iVF}). By using the Bianchi identity and the Ricci formula, we get
\begin{eqnarray*}
\nabla_p\nabla_p(V^qF_{qj})&=&V^q\nabla_p(-\nabla_qF_{jp}-\nabla_jF_{pq})
\\&=&-V^q(\nabla_q\nabla_pF_{jp}+F_{pq}F_{jp}-F_{jp}F_{pq})-V^q\nabla_p\nabla_jF_{pq}
\\&=&V^q\nabla_q(\frac{1}{2}x^pF_{pj})+[V^qF_{qp},F_{jp}]+\nabla_p\nabla_j(V^qF_{qp}),
\end{eqnarray*}
hence
\begin{eqnarray*}
L(V^qF_{qj})&=&\nabla_p\nabla_p(V^qF_{qj})-\nabla_p\nabla_j(V^qF_{qp})-[F_{pj},V^qF_{qp}]
\\&&-\frac{1}{2}x^p[\nabla_p(V^qF_{qj})-\nabla_j(V^qF_{qp})]
\\&=&V^q\nabla_q(\frac{1}{2}x^pF_{pj})-\frac{1}{2}x^p[\nabla_p(V^qF_{qj})-\nabla_j(V^qF_{qp})]
\\&=&\frac{1}{2}V^qF_{qj}+\frac{1}{2}x^pV^q(\nabla_qF_{pj}+\nabla_pF_{jq}+\nabla_jF_{qp})
\\&=&\frac{1}{2}V^qF_{qj}.
\end{eqnarray*}
\end{proof}

\begin{corollary}\label{gap}
Let $A$ be a homothetically shrinking soliton in $\mathcal{S}_{0,1}$.
If $|F|^2 < \frac{n}{2(n-1)} $, then $(E, A)$ is flat.
\end{corollary}

\begin{proof}
Note that  $J_j = \nabla^p F_{pj}=\frac{1}{2} x^p F_{pj}$. By integration by parts, we have
\begin{eqnarray*}
\int_{ \mathbb{R}^n} <(d^\nabla)^*d^\nabla J +  i_{ \frac{x}{2}} d^\nabla J ,  J >G dx = \int_{\mathbb{R}^n} |d^\nabla J|^2 Gdx.
\end{eqnarray*}
On the other hand by (\ref{J}), we have
\begin{eqnarray*}
\int_{ \mathbb{R}^n} <(d^\nabla)^*d^\nabla J +  i_{ \frac{x}{2}} d^\nabla J ,  J >G dx &=& \int_{ \mathbb{R}^n} <- LJ - \mathcal{R}(J), J>Gdx
\\&=& - \int_{\mathbb{R}^n}  |J|^2  G dx- \int_{ \mathbb{R}^n}  <[ F_{ij} , J_i] ,  J_j > G dx.
\end{eqnarray*}
For any $B, C\in so(r)$, we have $|[B,C]|\leq |B||C|$, see Lemma 2.30 in \cite{BourguignonLawson}.
Hence
\begin{eqnarray*}
|<[ F_{ij} , J_i] ,  J_j >|&\leq &|F_{ij}| |J_i| |J_j|=2\sum_{i<j}|F_{ij}| |J_i| |J_j|
\\&\leq&2\sqrt{\sum_{i<j}|F_{ij}|^2}\sqrt{\frac{1}{2}(|J|^4-\sum_k|J_k|^4)}
\\&\leq&2|F|\sqrt{\frac{1}{2}(1-\frac{1}{n})|J|^4}
\\&=&\sqrt{\frac{2(n-1)}{n}}|F| |J|^2
\end{eqnarray*}
and
$$ \int_{ \mathbb{R}^n}  |d^\nabla J|^2 Gdx \leq \int_{ \mathbb{R}^n}(\sqrt{\frac{2(n-1)}{n}}|F|-1) |J|^2  G dx.$$
If $ |F|^2 <  \frac{n}{2(n-1)}$, one then gets $J=0$.
Note that if $A\in \mathcal{S}_{0,1}$ has $J=0$, then for any $t_0>0$ we have $J=\frac{1}{2t_0}X$.
Hence Lemma \ref{identities} (a) holds for any $t_0>0$ and $F$ vanishes.
\end{proof}

\begin{theorem}\label{CharFstable}
Let $A$ be a homothetically shrinking soliton in $\mathcal{S}_{0,1}$.
Then it is F-stable if and only if the following properties are satisfied
\begin{itemize}
\item[(1)] $E_{-1} = \{ cJ, \quad c\in\mathbb{R} \} $;
\item[(2)]  $E_{- \frac{1}{2}} = \{ i_V F, \quad V \in \mathbb{R}^n \} $;
\item[(3)]  $E_{\lambda} = \{0\},  \mbox{ for any }  \lambda <0   \mbox{ and }  \lambda \neq -1,\, - \frac{1}{2} .$
\end{itemize}
\end{theorem}

\begin{proof}
Let  $\theta$ be a $\mathfrak{g}$-value $1$-form in $W^{2,2}_{G}$ of the form
$$ \theta = a  J + i_{W}F + \widetilde{\theta}, \quad a \in \mathbb{R}, W \in \mathbb{R}^n$$
and satisfying
$$\int_{\mathbb{R}^n} <\widetilde{\theta},J>Gdx=  \int_{ \mathbb{R}^n} <\widetilde{\theta},i_VF>Gdx = 0,  \quad \forall V \in \mathbb{R}^n.$$
Then it follows from Proposition \ref{secondvariation0}, Proposition \ref{twoeigenfunctions} and (\ref{Ladjoint}) that
\begin{eqnarray*}
\frac{1}{2} \mathcal{F}^{''}_{0,1} (q,V, \theta) &= &\int_{\mathbb{R}^n}<- L\theta - 2qJ - i_VF,\theta> Gdx
- \int_{\mathbb{R}^n} ( q^2 |J|^2 + \frac{1}{2}  |i_V F|^2) Gdx
\\&=& \int_{\mathbb{R}^n} <-a J -\frac{1}{2} i_W F - L \widetilde{\theta}, aJ + i_WF +\widetilde{\theta} > Gdx
\\&&+ \int_{ \mathbb{R}^n}  < - 2qJ - i_VF, aJ + i_W F +\widetilde{\theta}> G dx
\\&& -  \int_{ \mathbb{R}^n} ( q^2 |J|^2 + \frac{1}{2}  |i_V F|^2) G dx
\\&=&- (a + q)^2  \int_{\mathbb{R}^n}  |J|^2Gdx - \frac{1}{2}  \int_{ \mathbb{R}^n}|i_{V+W} F|^2   Gdx
\\&&+ \int_{ \mathbb{R}^n} <- L\widetilde{\theta} , \widetilde{\theta}> G dx
\end{eqnarray*}
Let $ q =-a$, $V= -W$, one has the equivalence.
\end{proof}

\section{Entropy and entropy-stability}

We now introduce $\lambda$-entropy of connections on the trivial $G$-vector bundle $E$ over $\mathbb{R}^n$.
We shall show that along the Yang-Mills flow, the entropy is non-increasing. We also prove that the entropy of a homothetically
shrinking soliton $A(x)\in \mathcal{S}_{x_0,t_0}$ is achieved exactly by $\mathcal{F}_{x_0,t_0}(A)$, provided that
$i_VF\neq 0$ for any non-zero vector $V\in \mathbb{R}^n$.

\begin{definition}
Let $A(x)$ be a connection on $E$. We define the entropy by
\begin{equation}\label{entropy}
\lambda(A)=\sup_{x_0\in \mathbb{R}^n,t_0>0}\mathcal{F}_{x_0,t_0}(A).
\end{equation}
\end{definition}

We first consider the invariance property of the entropy.
\begin{proposition}\label{invariantproperty}
The entropy $\lambda$ is invariant under translations and rescalings.
\end{proposition}
\begin{proof}
Let $A(x)$ be a connection on $E$. A translation of $A(x)$ is a new connection, denoted by $\widetilde{A}(x)$, of the form
$$\widetilde{A}_i(x)=A_i(x+x_1),$$
where $x_1$ is a point in $\mathbb{R}^n$. For any $x_0\in \mathbb{R}^n$ and $t_0>0$, we have
$$\mathcal{F}_{x_0-x_1,t_0}(\widetilde{A})=\mathcal{F}_{x_0,t_0}(A).$$
Hence
$$\lambda(\widetilde{A})=\lambda(A).$$

A rescaling of $A(x)$ is a new connection, denoted by $A^c(x)$ of the form
$$A_i^c(x)=c^{-1}A_i(c^{-1}x),$$
where $c$ is a positive number.
Then, by setting $y=c^{-1}x$, we have
\begin{eqnarray*}
\mathcal{F}_{cx_0,c^2t_0}(A^c)&=&(c^2t_0)^2\int_{\mathbb{R}^n}|F^c(x)|^2(4\pi c^2t_0)^{-\frac{n}{2}}e^{-\frac{|x-cx_0|^2}{4c^2t_0}}dx
\\&=&(c^2t_0)^2\int_{\mathbb{R}^n}c^{-4}|F(c^{-1}x)|^2(4\pi c^2t_0)^{-\frac{n}{2}}e^{-\frac{|x-cx_0|^2}{4c^2t_0}}dx
\\&=&(c^2t_0)^2\int_{\mathbb{R}^n}c^{-4}|F(y)|^2(4\pi c^2t_0)^{-\frac{n}{2}}e^{-\frac{|cy-cx_0|^2}{4c^2t_0}}c^ndy
\\&=&t_0^2\int_{\mathbb{R}^n}|F(y)|^2(4\pi t_0)^{-\frac{n}{2}}e^{-\frac{|y-x_0|^2}{4t_0}}dy
\\&=&\mathcal{F}_{x_0,t_0}(A).
\end{eqnarray*}
Hence
$$\lambda(A^c)=\lambda(A).$$
\end{proof}

In the case that $A(x)$ is a homothetically shrinking soliton, Proposition \ref{invariantproperty}
explains why in Theorem \ref{CharFstable} $J$ and $i_VF$ do not violate the F-stability.

\begin{proposition}
Let $A(x,t)$ be a solution to the Yang-Mills flow on $E$.
Then the entropy $\lambda(A(x,t))$ is non-increasing in $t$.
\end{proposition}

\begin{proof}
Let $t_1<t_2<T$. Here $T$ denotes the first singular time of the Yang-Mills flow.
By (\ref{entropy}), for any given $\epsilon>0$ there exists $(x_0,t_0)$ such that
\begin{equation}\label{approximateentropy}
\lambda(A(x,t_2))-\epsilon \leq \mathcal{F}_{x_0,t_0}(A(x,t_2)).
\end{equation}
Note that for any $c>0$ and $0\leq t<T$, we have
\begin{equation}\label{relation}
\mathcal{F}_{x_0,c}(A(x,t))=\Phi_{x_0,c+t}(A(x,t)).
\end{equation}
By (\ref{relation}), the monotonicity formula (\ref{monotonicityformula}),  and the definition of entropy, we have
\begin{eqnarray*}
\mathcal{F}_{x_0,t_0}(A(x,t_2))&=&\Phi_{x_0,t_0+t_2}(A(x,t_2))
\\&\leq& \Phi_{x_0,t_0+t_2}(A(x,t_1))= \mathcal{F}_{x_0,t_0+t_2-t_1}(A(x,t_1))
\\&\leq& \lambda(A(x,t_1)).
\end{eqnarray*}
Together with (\ref{approximateentropy}), we see that
$\lambda(A(x,t_2))\leq \lambda(A(x,t_1)).$
\end{proof}

\begin{definition}
A homothetically shrinking soliton $A(x)$ is called entropy-stable if it is a local minimum of the entropy, among
all perturbations $\widetilde{A}(x)$ such that $||\widetilde{A}-A||_{C^1}$ is sufficiently small.
\end{definition}

In general the entropy $\lambda(A)$ is not attained by any $\mathcal{F}_{x_0,t_0}(A)$.
However if $A\in \mathcal{S}_{x_0,t_0}$ and $i_VF\neq 0$ for any $V\in \mathbb{R}^n$, we will show that
$\lambda(A)$ is attained exactly by $\mathcal{F}_{x_0,t_0}(A)$.
We first examine the geometric meaning of $i_VF=0$ in the case that $A(x)$ is a homothetically shrinking solition.

\begin{proposition}
If $A(x)$ is a homothetically shrinking soliton satisfying $i_VF=0$ for some non-zero vector $V$,
then $A(x)$ is defined on a hyperplane perpendicular to $V$.
\end{proposition}

\begin{proof}
Without loss of generality we assume $A(x)$ is centered at $(0,1)$ and let $A(x,t)$ be the homothetically
shrinking Yang-Mills flow with $A(x,0)=A(x)$.
In the exponential gauge, i.e. a gauge such that $x^jA_j(x)=0$, we have for any $t<1$ and $\lambda>0$ that
\begin{equation}\label{A(x,t)}
A_j(x,t)=\lambda A_j(\lambda x,\lambda^2(t-1)+1)=\frac{1}{\sqrt{1-t}}A_j(\frac{x}{\sqrt{1-t}},0)=\frac{1}{\sqrt{1-t}}A_j(\frac{x}{\sqrt{1-t}})
\end{equation}
and
\begin{equation}\label{F(x,t)}
F_{ij}(x,t)=\frac{1}{1-t}F_{ij}(\frac{x}{\sqrt{1-t}}).
\end{equation}
Moreover the exponential gauge is uniform for all $t<1$, i.e. $x^jA_j(x,t)=0$.

By assumption we have $i_VF(x)=0$. For simplicity let $V=\frac{\partial}{\partial x^l}$. Then by (\ref{F(x,t)}) we have
$$F_{jl}(x,t)=0, \quad \forall j.$$
Note that $A(x,t)$ is a homothetically shrinking Yang-Mills flow, hence
$$J_l(x,t)=\frac{1}{2(1-t)}x^jF_{jl}(x,t)=0.$$
Then
$$\frac{\partial}{\partial t}A_l(x,t)=J_l(x,t)=0.$$
In particular,
$$A_l(x,t')=A_l(x,t), \quad \forall t,t'<1.$$
Then by (\ref{A(x,t)}), we have for any $\lambda>0$ and $t<1$ that
$$A_l(x,t)=\lambda A_l(\lambda x,\lambda^2(t-1)+1)=\lambda A_l(\lambda x,t).$$
Letting $\lambda\rightarrow 0$, we see that
\begin{equation}\label{Al}
A_l(x,t)= 0.
\end{equation}
Note that
$$0=F_{l j}(x,t)=\partial_l A_j-\partial_jA_l+A_l A_j-A_jA_l=\partial_l A_j(x,t),$$
so for any $j$, we have
$$\partial_l A_j(x,t)=0$$
and
\begin{equation}\label{Aj}
A_j(x+cV,t)=A_j(x,t), \quad \forall c\in \mathbb{R}.
\end{equation}
For example if $V=\frac{\partial}{\partial x^n}$, then by (\ref{Al}) and (\ref{Aj}) we have
$$A(x^1,\cdots,x^{n-1},x^n,t)=A_1(x^1,\cdots,x^{n-1},0,t)dx^1+\cdots+A_{n-1}(x^1,\cdots,x^{n-1},0,t)dx^{n-1}.$$

In particular for $V=\frac{\partial}{\partial x^n}$ and in the exponential gauge, we have
\begin{equation}\label{splitting}
A(x^1,\cdots,x^{n-1},x^n)=A_1(x^1,\cdots,x^{n-1},0)dx^1+\cdots+A_{n-1}(x^1,\cdots,x^{n-1},0)dx^{n-1}.
\end{equation}
This means that $A(x)$ is defined on a hyperplane perpendicular to $V$,
i.e. $A(x)$ descends to a trivial $G$-vector bundle over a hyperplane $V^\perp$.
\end{proof}

The following Proposition is analogous to a corresponding result for self-shrinkers of the mean curvature flow,
see \cite{ColdingMinicozzi}.
We follow closely the arguments given in \cite{ColdingMinicozzi}.

\begin{proposition}\label{entropyobtained}
Let $A(x)$ be a homothetically shrinking soliton centered at $(0,1)$ such that $i_VF\neq 0$ for any non-zero $V$.
Then the function $(x_0,t_0)\mapsto \mathcal{F}_{x_0,t_0}(A)$ attains its strict maximum at $(0,1)$.
In fact for any given $\epsilon>0$, there exists a constant $\delta>0$ such that
\begin{equation}\label{strictmaximum}
\sup\{\mathcal{F}_{x_0,t_0}(A): |x_0|+|\log t_0|\geq \epsilon\}<\lambda(A)-\delta.
\end{equation}
In particular, the entropy of $A$ is achieved by $\mathcal{F}_{0,1}(A)$.
\end{proposition}

\begin{proof}
We first show that $(0,1)$ is a local maximum of the function $(x_0,t_0)\mapsto \mathcal{F}_{x_0,t_0}(A)$.
That is to show
$$\mathcal{F}'_{0,1}(q,V,0)=0, \quad \forall q, V,$$
and
$$\mathcal{F}''_{0,1}(q,V,0)<0, \quad \forall (q,V)\neq (0,0).$$
In fact by the first variation formula (\ref{firstvariation}) and Lemma \ref{identities} (a, b), we have
$$\frac{d}{ds}|_{s=0}\mathcal{F}_{x_s,t_s}(A)=\mathcal{F}'_{0,1}(q,V,0)=0.$$
Let $x_s=sV, t_s=1+sq$. Note that $J\neq 0$, otherwise $F$ would be vanishing, as showed in the proof of Corollary \ref{gap},
which violates the assumption that $i_VF\neq 0$ for any non-zero $V$.
Then by the second variation formula (\ref{secondvariation}), we have for any $(q,V)\neq (0,0)$ that
$$\frac{1}{2}\mathcal{F}''_{0,1}(q,V,0)=-\int_{\mathbb{R}^n}(q^2|J|^2+\frac{1}{2}|i_VF|^2)Gdx< 0.$$

For any fixed $(y,T)$, where $y\in \mathbb{R}^n$ and $T>0$, we set
$$x_s=sy,\quad  t_s=1+(T-1)s^2.$$
Note that $(x_s,t_s), s\in [0,1]$, is a path from $(0,1)$ to $(y,T)$.
Let
$$g(s)=\mathcal{F}_{x_s,t_s}(A).$$
The remaining of the proof is to show that $g'(s)\leq 0$ for $s\in [0,1]$.

By the first variation formula (\ref{firstvariation}), we have
\begin{eqnarray*}
g'(s)&=&\int_{\mathbb{R}^n}\dot{t}_s(\frac{4-n}{2}t_s+\frac{1}{4}|x-x_s|^2)|F|^2G_s(x)dx
\\&&+\int_{\mathbb{R}^n}\frac{1}{2}t_s<\dot{x}_s,x-x_s>|F|^2G_s(x)dx.
\end{eqnarray*}
In the same way as in the proof of Lemma \ref{identities}, for vector fields $\varphi$ on $\mathbb{R}^n$ we have
\begin{eqnarray*}
&&\int_{\mathbb{R}^n}\varphi^p(x-x_s)^p|F|^2G_s(x)dx
\\&=&\int_{\mathbb{R}^n}[2t_s\partial_p(\varphi^p)|F|^2-4t_s\partial_i\varphi^pF_{pj\beta}^\alpha F_{ij\beta}^\alpha]G_s dx
\\&&-\int_{\mathbb{R}^n}4t_s\varphi^pF_{pj\beta}^\alpha (J_{j\beta}^\alpha -\frac{1}{2t_s}X_{j\beta}^\alpha)G_s dx,
\end{eqnarray*}
where $$X_{j\beta}^\alpha=(x-x_s)^pF_{pj\beta}^\alpha.$$
Taking $\varphi=\frac{\partial}{\partial x^p}$ and noting that $J_j=\frac{x^p}{2}F_{pj}$, we get
\begin{eqnarray*}
\int_{\mathbb{R}^n}(x-x_s)^p|F|^2G_s(x)dx
&=&-\int_{\mathbb{R}^n}4t_sF_{pj\beta}^\alpha (J_{j\beta}^\alpha -\frac{1}{2t_s}X_{j\beta}^\alpha)G_s dx,
\\&=&-\int_{\mathbb{R}^n}4t_sF_{pj\beta}^\alpha (\frac{1}{2}x^i -\frac{1}{2t_s}(x-x_s)^i)F_{ij\beta}^\alpha G_s dx.
\end{eqnarray*}
Taking $\varphi(x)=x-x_s$, we get
\begin{eqnarray*}
&&\int_{\mathbb{R}^n}|x-x_s|^2|F|^2G_sdx
\\&=&\int_{\mathbb{R}^n}[2(n-4)t_s|F|^2]G_s dx
-\int_{\mathbb{R}^n}4t_sX_{j\beta}^\alpha (J_{j\beta}^\alpha -\frac{1}{2t_s}X_{j\beta}^\alpha)G_s dx
\\&=&\int_{\mathbb{R}^n}[2(n-4)t_s|F|^2+2|X|^2]G_s dx
-\int_{\mathbb{R}^n}2t_sX_{j\beta}^\alpha x^i F_{ij\beta}^\alpha G_s dx.
\end{eqnarray*}
Hence we have
\begin{eqnarray*}
g'(s)&=&-\int_{\mathbb{R}^n}\frac{n-4}{2}t_s\dot{t}_s|F|^2G_s(x)dx
\\&&+\frac{1}{4}\dot{t}_s[\int_{\mathbb{R}^n}[2(n-4)t_s|F|^2+2|X|^2]G_s dx
-\int_{\mathbb{R}^n}2t_sX_{j\beta}^\alpha x^i F_{ij\beta}^\alpha G_s dx]
\\&&-t_sy^p\int_{\mathbb{R}^n}t_sF_{pj\beta}^\alpha (x^i -\frac{1}{t_s}(x-x_s)^i)F_{ij\beta}^\alpha G_s dx
\\&=&\frac{1}{2}\dot{t}_s[\int_{\mathbb{R}^n}|X|^2G_s dx
-\int_{\mathbb{R}^n}t_sX_{j\beta}^\alpha x^i F_{ij\beta}^\alpha G_s dx]
\\&&-t_sy^p\int_{\mathbb{R}^n}t_sF_{pj\beta}^\alpha (x^i -\frac{1}{t_s}(x-x_s)^i)F_{ij\beta}^\alpha G_s dx.
\end{eqnarray*}
Set $z=x-x_s=x-sy$. We have $x=z+sy$ and $X_j=z^iF_{ij}$. Then we get
\begin{eqnarray*}
g'(s)&=&\frac{1}{2}\dot{t}_s[\int_{\mathbb{R}^n}(1-t_s)|X|^2G_s dx
-\int_{\mathbb{R}^n}t_sX_{j\beta}^\alpha sy^i F_{ij\beta}^\alpha G_s dx]
\\&&-t_sy^p\int_{\mathbb{R}^n}t_sF_{pj\beta}^\alpha (z^i+sy^i -\frac{1}{t_s}z^i)F_{ij\beta}^\alpha G_s dx
\\&=&\frac{1}{2}\dot{t}_s[\int_{\mathbb{R}^n}(1-t_s)|X|^2G_s dx
-\int_{\mathbb{R}^n}t_sX_{j\beta}^\alpha sy^i F_{ij\beta}^\alpha G_s dx]
\\&&-t_sy^p\int_{\mathbb{R}^n}(t_s-1)F_{pj\beta}^\alpha X_{j\beta}^\alpha G_s dx
-t_s^2\int_{\mathbb{R}^n}sy^pF_{pj\beta}^\alpha y^iF_{ij\beta}^\alpha G_s dx
\\&=&\frac{1}{2}\dot{t}_s(1-t_s)\int_{\mathbb{R}^n}|X|^2G_s dx-(\frac{1}{2}s\dot{t}_st_s+t_s(t_s-1) )\int_{\mathbb{R}^n}<X_j,y^i F_{ij}> G_s dx
\\&&-st_s^2\int_{\mathbb{R}^n}|y^iF_{ij}|^2 G_s dx.
\end{eqnarray*}
For $t_s=1+(T-1)s^2$, we have
\begin{eqnarray*}
g'(s)&=&-s[(T-1)^2s^2\int_{\mathbb{R}^n}|X|^2G_s dx+2(T-1)st_s\int_{\mathbb{R}^n}<X_j,y^i F_{ij}> G_s dx
\\&&+t_s^2\int_{\mathbb{R}^n}|y^iF_{ij}|^2 G_s dx]
\\&=&-s\int_{\mathbb{R}^n}|(T-1)sX_j+t_sy^iF_{ij}|^2G_s dx
\\&\leq&0.
\end{eqnarray*}

\end{proof}

\section{entropy-stability and F-stability}

In this section we shall show that the entropy-stability of a homothetically shrinking soliton such that
$i_VF\neq 0$ for any non-zero $V$ implies F-stability.

\begin{theorem}
Let $A(x)$ be a homothetically shrinking soliton in $\mathcal{S}_{0,1}$ such that $i_VF\neq 0$ for any non-zero $V$.
If $A(x)$ is entropy-stable, then it is F-stable.
\end{theorem}

\begin{proof}
We argue by contradiction. Assume that $A(x)$ is F-unstable.
By the definition of F-stability there exists a $1$-parameter family of connections
$A_s(x), s\in [-\epsilon, \epsilon]$, with $\theta_s(x):=\frac{d}{ds}A_s(x)\in C_0^\infty$, such that
for any deformation $(x_s, t_s)$ of $(x_0=0, t_0=1)$, we have
\begin{equation}\label{Funstable}
\frac{d^2}{ds^2}|_{s=0}\mathcal{F}_{x_s,t_s}(A_s)<0.
\end{equation}
We start from this to show that $A$ is entropy-unstable.
Let
$$H: \mathbb{R}^n\times \mathbb{R}^+\times [-\epsilon,\epsilon],\quad H(y,T,s)=\mathcal{F}_{y,T}(A_s).$$
In fact we will show that there exists $\epsilon_0>0$ such that for $s$ with $0< |s|\leq \epsilon_0$,
\begin{equation}\label{lessentropy}
\sup_{y,T}H(y,T,s)<H(0,1,0).
\end{equation}
Hence for $s$ with $0< |s|\leq \epsilon_0$, $\lambda(A_s)<\lambda(A)$, which contradicts with our assumption.

\bigskip
Step 1. We prove that there exists $\epsilon_1>0$ such that for any $s$ with $0<|s|\leq \epsilon_1$,
\begin{equation}\label{nearbypoints}
\sup \{H(y,T,s): |y|\leq \epsilon_1, |\log T|\leq \epsilon_1\}<H(0,1,0).
\end{equation}
By the assumption that $A(x)\in \mathcal{S}_{0,1}$ and Corollary \ref{criticalofentropy}, we have
$$\nabla H(0,1,0)=0.$$
For any $y\in \mathbb{R}^n, a\in \mathbb{R}$ and $b\in \mathbb{R}$, $(sy,1+as,bs)$ is a curve
through $(0,1,0)$. In case of $b\neq 0$, by  (\ref{Funstable}) we have
\begin{eqnarray*}
\frac{d^2H}{ds^2}|_{s=0}(sy,1+as,bs)&=&\frac{d^2}{ds^2}|_{s=0}\mathcal{F}_{sy,1+as}(A_{bs})
\\&=&b^2\frac{d^2}{ds^2}|_{s=0}\mathcal{F}_{\frac{s}{b}y,1+a\frac{s}{b}}(A_{s})
\\&< &0.
\end{eqnarray*}
For $b=0$ and $(a,y)\neq (0,0)$, we have
\begin{eqnarray*}
\frac{d^2H}{ds^2}|_{s=0}(sy,1+as,0)&=&\frac{d^2}{ds^2}|_{s=0}\mathcal{F}_{sy,1+as}(A)
\\&=&-2\int_{\mathbb{R}^n}(a^2|J|^2+\frac{1}{2}|i_yF|^2)Gdx
\\&< &0,
\end{eqnarray*}
where we used the assumption that $i_yF\neq 0$ for $y\neq 0$ and its implication that $J\neq 0$.
Hence the Hessian of $H$ at $(0,1,0)$ is negative definite and $H$ has a local strict maximum at $(0,1,0)$.
Thus there exists $\epsilon_1\in (0,\epsilon]$ such that if $0<|y|+|\log T|+|s|\leq 3\epsilon_1$, then
$H(y,T,s)<H(0,1,0).$
In particular for any $s$ with $0<|s|\leq \epsilon_1$, we have
$$\sup \{H(y,T,s): |y|\leq \epsilon_1, |\log T|\leq \epsilon_1\}<H(0,1,0).$$

\bigskip

Step 2. We prove that there exists $R_0>0$ such that
\begin{equation}\label{largey}
\sup_{T,s} H(y,T,s)<H(0,1,0), \quad \text{for} \quad |y|\geq R_0.
\end{equation}
Denote the support of $\theta_s$ by $\Omega_s$ and $\Omega=\bigcup_{s\in [-\epsilon,\epsilon]}\Omega_s$.
Then on $\mathbb{R}^n\setminus\Omega$, $F_s=F$. Hence
\begin{eqnarray*}
H(y,T,s)&=&T^2\int_{\Omega}|F_s|^2(4\pi T)^{-\frac{n}{2}}e^{-\frac{|x-y|^2}{4T}}dx
+T^2\int_{\mathbb{R}^n\setminus\Omega}|F|^2(4\pi T)^{-\frac{n}{2}}e^{-\frac{|x-y|^2}{4T}}dx
\\&\leq &T^2\int_{\Omega}|F_s|^2(4\pi T)^{-\frac{n}{2}}e^{-\frac{|x-y|^2}{4T}}dx+H(y,T,0).
\end{eqnarray*}
Note that for $|y|\geq 1$, there exists $\delta>0$ such that $H(y,T,0)\leq H(0,1,0)-\delta$, see Proposition \ref{entropyobtained}.
Let $M=\sup\{|F_s(x)|^2: s\in [-\epsilon,\epsilon], x\in \mathbb{R}^n\}$,
$D=\sup_{x\in\Omega}|x|$ and $|\Omega|=\int_\Omega dx$. Then for $|y|\geq D+R$ with $R\geq 1$, we have
\begin{eqnarray*}
H(y,T,s)&\leq & M|\Omega|T^2(4\pi T)^{-\frac{n}{2}}e^{-\frac{R^2}{4T}}+H(0,1,0)-\delta.
\end{eqnarray*}
Let $f(r)=r^{-\frac{n-4}{2}}e^{-\frac{1}{4r}}$, $r>0$, which is uniformly bounded.
Note that $n\geq 5$.
Then as $R\rightarrow \infty$, $T^{\frac{4-n}{2}}e^{-\frac{R^2}{4T}}=f(\frac{T}{R^2})R^{4-n}\rightarrow 0$, uniformly in $T>0$.
Hence we can choose sufficiently large $R$ such that for $|y|\geq D+R:=R_0$, we have
$H(y,T,s)\leq H(0,1,0)-\frac{\delta}{2}$.

\bigskip
Step 3. We prove that exists $T_0>0$ such that
\begin{equation}\label{largelogT}
\sup_{y,s}H(y,T,s)<H(0,1,0), \quad \text{for} \quad |\log T|\geq T_0.
\end{equation}
We first consider the case that $T$ is large. Note that for any $T>0$,
\begin{eqnarray*}
H(y,T,s)&=&T^2\int_{\Omega}|F_s|^2(4\pi T)^{-\frac{n}{2}}e^{-\frac{|x-y|^2}{4T}}dx
+T^2\int_{\mathbb{R}^n\setminus\Omega}|F|^2(4\pi T)^{-\frac{n}{2}}e^{-\frac{|x-y|^2}{4T}}dx
\\&\leq &T^2\int_{\Omega}|F_s|^2(4\pi T)^{-\frac{n}{2}}e^{-\frac{|x-y|^2}{4T}}dx+H(y,T,0)
\\&\leq& M|\Omega|T^2(4\pi T)^{-\frac{n}{2}}+H(y,T,0).
\end{eqnarray*}
By Proposition \ref{entropyobtained}, there exists $\delta>0$ such that $H(y,T,0)\leq H(0,1,0)-\delta$ when $T\geq 2$.
Hence there exists $T_1\geq 2$ such that
\begin{equation}\label{largeT}
H(y,T,s)\leq H(0,1,0)-\frac{\delta}{2}, \quad \text{for} \quad T\geq T_1.
\end{equation}
Note that $M=\sup\{|F_s(x)|^2: s\in [-\epsilon,\epsilon], x\in \mathbb{R}^n\}$. Hence for any $T>0$, we have
$$H(y,T,s)=\mathcal{F}_{y,T}(A_s)=T^2\int_{\mathbb{R}^n}|F_s|^2(4\pi T)^{-\frac{n}{2}}e^{-\frac{|x-y|^2}{4T}}dx\leq MT^2.$$
Thus there exists $T_2>0$ such that
\begin{equation}\label{smallT}
\sup_{y,s}H(y,T,s)<H(0,1,0), \quad \text{for} \quad T\leq T_2.
\end{equation}
Combing (\ref{largeT}) and (\ref{smallT}), we get (\ref{largelogT}).

\bigskip
Step 4. Set $$U=\{(y,T): |y|\leq R_0, |\log T|\leq T_0\}\setminus \{(y, T): |y|<\epsilon_1, |\log T|<\epsilon_1\}.$$
We now prove that there exists $\epsilon_0\leq \epsilon_1$  such that for any $s$ with $|s|\leq \epsilon_0$,
\begin{equation}\label{compactset}
\sup\{H(y,T,s): (y,T)\in U  \}<H(0,1,0).
\end{equation}
Note that $U$ is a compact set which does not contain $(0,1)$.
By Proposition \ref{entropyobtained}, there exists $\delta>0$ such that
$$\sup_UH(y,T,0)\leq H(0,1,0)-\delta.$$
By the first variation formula (\ref{firstvariation}) of the $\mathcal{F}$-functional, we have
$$\frac{d}{d s}H(y,T,s)=-2T^2\int_{\mathbb{R}^n}<J(A_s)-\frac{1}{2T}i_{x-y}F(A_s),\theta_s>G_{y,T}(x)dx.$$
Since $\theta_s$ is compactly supported,  $\partial_sH$ is continuous in all three variables $y, T$, and $s$.
Therefore there exists $0<\epsilon_0\leq \epsilon_1$ such that if $|s|\leq \epsilon_0$, then
$$\sup_UH(y,T,s)\leq H(0,1,0)-\frac{\delta}{2}.$$
This proves (\ref{compactset}).
Combining (\ref{nearbypoints}), (\ref{largey}), (\ref{largelogT}) and (\ref{compactset}), we get (\ref{lessentropy}) and complete the proof.

\end{proof}

\end{document}